\pdfoutput=1
\RequirePackage{ifpdf}
\ifpdf % We are running pdfTeX in pdf mode
\documentclass[pdftex]{sigma}
\else
\documentclass{sigma}
\fi

\numberwithin{equation}{section}

\newtheorem{Theorem}{Theorem}[section]
\newtheorem{Corollary}[Theorem]{Corollary}
\newtheorem{Lemma}[Theorem]{Lemma}
\newtheorem{Proposition}[Theorem]{Proposition}
 { \theoremstyle{definition}
\newtheorem{Definition}[Theorem]{Definition}
\newtheorem{Example}[Theorem]{Example}
\newtheorem{Remark}[Theorem]{Remark} }

\DeclareMathOperator\C{\mathbb C}
\DeclareMathOperator\Z{\mathbb Z}

\begin{document}

\allowdisplaybreaks

\newcommand{\arXivNumber}{1702.08060}

\renewcommand{\thefootnote}{}

\renewcommand{\PaperNumber}{132}

\FirstPageHeading

\ShortArticleName{Elliptic Dynamical Quantum Groups and Equivariant Elliptic Cohomology}

\ArticleName{Elliptic Dynamical Quantum Groups\\ and Equivariant Elliptic Cohomology\footnote{This paper is a~contribution to the Special Issue on Elliptic Hypergeometric Functions and Their Applications. The full collection is available at \href{https://www.emis.de/journals/SIGMA/EHF2017.html}{https://www.emis.de/journals/SIGMA/EHF2017.html}}}

\Author{Giovanni FELDER~$^\dag$, Rich\'ard RIM\'ANYI~$^\ddag$ and Alexander VARCHENKO~$^\ddag$}

\AuthorNameForHeading{G.~Felder, R.~Rim\'anyi and A.~Varchenko}

\Address{$^\dag$~Department of Mathematics, ETH Z\"urich, 8092 Z\"urich, Switzerland}
\EmailD{\href{mailto:felder@math.ethz.ch}{felder@math.ethz.ch}}

\Address{$^\ddag$~Department of Mathematics, University of North Carolina at Chapel Hill,\\
\hphantom{$^\ddag$}~Chapel Hill, NC 27599-3250, USA}
\EmailD{\href{mailto:rimanyi@email.unc.edu}{rimanyi@email.unc.edu}, \href{mailto:anv@email.unc.edu}{anv@email.unc.edu}}

\ArticleDates{Received April 30, 2018, in final form December 12, 2018; Published online December 21, 2018}

\Abstract{We define an elliptic version of the stable envelope of Maulik and Okounkov for the equivariant elliptic cohomology of cotangent bundles of Grassmannians. It is a version of the construction proposed by Aganagic and Okounkov and is based on weight functions and shuffle products. We construct an action of the dynamical elliptic quantum group associated with $\mathfrak{gl}_2$ on the equivariant elliptic cohomology of the union of cotangent bundles of Grassmannians. The generators of the elliptic quantum groups act as difference operators on sections of admissible bundles, a notion introduced in this paper.}


\Classification{17B37; 55N34; 32C35; 55R40}

\renewcommand{\thefootnote}{\arabic{footnote}}
\setcounter{footnote}{0}

\section{Introduction}\label{s-1}
Maulik and Okounkov \cite{MO} have set up a program to realize representation theory of quantum groups of various kinds on torus equivariant (generalized) cohomology of Nakajima varieties. A~central role is played by the stable envelopes, which are maps from the equivariant cohomology of the fixed point set of the torus action to the equivariant cohomology of the variety. Stable envelopes depend on the choice of a chamber (a connected component of the complement of an arrangement of real hyperplanes) and different chambers are related by $R$-matrices of the corresponding quantum groups. The basic example of a Nakajima variety is the cotangent bundle of the Grassmannian $\mathrm{Gr}(k,n)$ of $k$-planes in $\mathbb C^n$. The torus is $T=U(1)^{n}\times U(1)$, with~$U(1)^n$ acting by diagonal matrices on $\mathbb C^n$ and $U(1)$ acting by multiplication on the cotangent spaces. Then the Yangian $Y(\mathfrak{gl}_2)$ acts on $H_T(\sqcup_{k=0}^nT^*\mathrm{Gr}(k,n))$ and the action of generators is described geometrically by correspondences. It turns out that this representation is isomorphic to the tensor products of $n$ evaluation vector representations with the equivariant parameters of~$U(1)^n$ as evaluation points and the equivariant parameter of $U(1)$ as the deformation parameter of the quantum group. The choice of a chamber is the same as the choice of an ordering of the factors in the tensor product. The same holds for the affine quantum universal enveloping algebra $U_q(\widehat {\mathfrak{gl}}_2)$ if we replace equivariant cohomology by equivariant $K$-theory. As was shown in \cite{GRTV, RTV}, the stable envelopes, which realize the isomorphisms, are given by the {\em weight functions}, which originally appeared in the theory of integral representations of solutions of the Knizhnik--Zamolodchikov equation, see \cite{TV1,TV2}. Their special values form transition matrices from the tensor basis to a~basis of eigenvectors for the Gelfand--Zetlin commutative subalgebra.

The recent preprint \cite{AO} of Aganagic and Okounkov suggests that the same picture should hold for equivariant elliptic cohomology and elliptic dynamical quantum groups and this is the subject of this paper. The authors of~\cite{AO} define an elliptic version of the stable envelopes and show, in the example of the cotangent bundle of a~projective space, stable envelopes corresponding to different orderings are related to the fundamental elliptic dynamical $R$-matrices of the elliptic dynamical quantum group $E_{\tau,y}(\mathfrak{gl}_2)$. Our paper is an attempt to understand the elliptic stable envelope in the case of cotangent bundles of Grassmannians. In particular we give a precise description of the space in which the stable envelope takes its values. Our construction of stable envelopes is based on elliptic weight functions. In Appendix~\ref{sec:appendix} we also give a geometric characterization, in terms of pull-backs to the cohomology of fixed points, in the spirit of~\cite{MO}.

While our work is inspired by \cite{AO}, we do not know whether the two constructions are equivalent or not. The interesting project of understanding the exact relation between our construction and the construction of Aganagic--Okounkov requires more work.

Compared to equivariant cohomology and $K$-theory, two new features arise in the elliptic case. The first new feature is the occurrence of an additional variable, the dynamical parameter, in the elliptic quantum group. It also appears in \cite{AO}, under the name of K\"ahler parameter, in an extended version of the elliptic cohomology of Nakajima varieties. The second is a general feature of elliptic cohomology: while $T$-equivariant cohomology and $K$-theory are contravariant functors from $T$-spaces to supercommutative algebras, and can thus be thought of as covariant functors to affine superschemes,\footnote{The reader may safely ignore the super prefixes, as we only consider spaces with trivial odd cohomology, for which one has strictly commutative algebras.} in the elliptic case only the description as covariant functor to (typically non-affine) superschemes generalizes straightforwardly.

Our main result is a construction of an action of the elliptic quantum group associated with~$\mathfrak{gl}_2$ on the extended equivariant elliptic cohomology scheme $\hat E_T(X_{n})$ of the union $X_n=\sqcup_{k=0}^n X_{k,n}$ of cotangent bundles $X_{k,n}=T^*\mathrm{Gr}(k,n)$ of Grassmannians. The meaning of this is that we define a representation of the operator algebra of the quantum group by difference operators acting on sections of a class of line bundles on the extended elliptic cohomology scheme, which we call admissible bundles: up to a twist by a fixed line bundle, admissible bundles on $\hat E_T(X_{k,n})$ are pull-backs of bundles on $\hat E_{U(n)\times U(1)}(\mathrm{pt})$ (by functoriality there is a~map corresponding to the map to a point and the inclusion of the Cartan subalgebra $T\to U(n)\times U(1)$). The claim is that there is a representation of the elliptic quantum group by operators mapping sections of admissible bundles to sections of admissible bundles.

This paper may be considered as an elliptic version of the paper \cite{RV2} where analogous constructions are developed for the rational dynamical quantum group $E_y(\mathfrak{gl}_2)$.

\subsection*{Notation} For a positive integer $n$, we set $[n]=\{1,\dots,n\}$. It $K$ is a subset of $[n]$ we denote by $|K|$ its cardinality and by $\bar K$ its complement. Throughout the paper we fix $\tau$ in the upper half plane and consider the complex elliptic curve $E=\mathbb C/(\mathbb Z+\tau \mathbb Z)$. The odd Jacobi theta function
\begin{gather}\label{e-theta}
 \theta(z)=\frac{\sin \pi z}{\pi}\prod_{j=1}^\infty\frac{\big(1-q^j{\rm e}^{2\pi{\rm i}z}\big)\big(1-q^j{\rm e}^{-2\pi{\rm i}z}\big)}{\big(1-q^j\big)^2},\qquad q={\rm e}^{2\pi{\rm i}\tau},
\end{gather}
is normalized to have derivative 1 at $0$. It is an entire odd function with simple zeros at $\mathbb Z+\tau\mathbb Z$, obeying $\theta(z+1)=-\theta(z)$ and
\begin{gather*}
 \theta(z+\tau)=-{\rm e}^{-\pi {\rm i}\tau-2\pi{\rm i}z}\theta(z).
\end{gather*}

\section[Dynamical $R$-matrices and elliptic quantum groups]{Dynamical $\boldsymbol{R}$-matrices and elliptic quantum groups}\label{s-2}
\subsection{Dynamical Yang--Baxter equation}\label{ss-2.1}
Let $\mathfrak h$ be a complex abelian Lie algebra and $V$ an $\mathfrak h$-module with a weight decomposition $V=\oplus_{\mu\in\mathfrak h^*}V_\mu$ and finite dimensional weight spaces $V_\mu$. A dynamical $R$-matrix with values in $\operatorname{End}_{\mathfrak h}(V\otimes V)$ is a meromorphic function $(z,y,\lambda)\mapsto R(z,y,\lambda)\in\operatorname{End}_{\mathfrak h}(V\otimes V)$ of the spectral parameter $z\in\mathbb C$, the deformation parameter $y\in\mathbb C$ and the dynamical parameter $\lambda\in\mathfrak h^*$, obeying the dynamical Yang--Baxter equation
\begin{gather}
 R\big(z,y,\lambda-y h^{(3)}\big)^{(12)}R(z+w,y,\lambda)^{(13)} R\big(w,y,\lambda-y h^{(1)}\big)^{(23)} \notag\\
 \qquad {}= R(w,y,\lambda)^{(23)}R\big(z+w,y,\lambda-y h^{(2)}\big)^{(13)} R\big(z,y,\lambda-y h^{(3)}\big)^{(12)}\label{e-YBE}
\end{gather}
in $\operatorname{End}(V\otimes V\otimes V)$ and the inversion relation
\begin{gather}\label{e-inversion}
 R(z,y,\lambda)^{(12)}R(-z,y,\lambda)^{(21)}=\mathrm{Id}
\end{gather}
in $\operatorname{End}(V\otimes V)$. The superscripts indicate the factors in the tensor product on which the endomorphisms act non-trivially and $h$ is the element in $\mathfrak h^*\otimes \operatorname{End}(V)$ defined by the action of~$\mathfrak h$: for example $R\big(z,y,\lambda-y h^{(3)}\big)^{(12)}$ acts as $R(z,y,\lambda-y\mu_3)\otimes \mathrm{Id}$ on $V_{\mu_1}\otimes V_{\mu_2}\otimes V_{\mu_3}$.
\begin{Example}[\cite{Felder}]\label{example-1} Let $\mathfrak h\simeq\mathbb C^N$ be the Cartan subalgebra of diagonal matrices in $\mathfrak{gl}_N(\mathbb C)$. Let $V=\oplus_{i=1}^N V_{\epsilon_i}$ the vector representation with weights $\epsilon_i(x)=x_i$, $x\in\mathfrak h$ and one-dimensional weight spaces. Let $E_{ij}$ be the $N\times N$ matrix with entry $1$ at $(i,j)$ and $0$ elsewhere. The elliptic dynamical $R$-matrix
 for $\mathfrak{gl}_N$ is\footnote{We use the convention of~\cite{FelderVarchenko}. This $R$-matrix is obtained from the one introduced in \cite{Felder} by substituting $y=-2\eta$ and replacing $z$ by $-z$.}
 \begin{gather*}
 R(z,y,\lambda) =\sum_{i=1}^NE_{ii}\otimes E_{ii}+\sum_{i\neq j}\alpha(z,y,\lambda_i-\lambda_j) E_{ii}\otimes E_{jj} +\sum_{i\neq j}\beta(z,y,\lambda_i-\lambda_j) E_{ij}\otimes E_{ji},
 \end{gather*}
 where
 \begin{gather*}
 \alpha(z,y,\lambda)= \frac{\theta(z)\theta(\lambda+y)}{\theta(z-y)\theta(\lambda)}, \qquad \beta(z,y,\lambda)= -\frac{\theta(z+\lambda)\theta(y)}{\theta(z-y)\theta(\lambda)}.
 \end{gather*}
 It is a deformation of the trivial $R$-matrix $R(z,0,\lambda)=\mathrm{id}_{V\otimes V}$.
\end{Example}
A dynamical $R$-matrix defines a representation of the symmetric group $S_n$ on the space of meromorphic functions of $(z_1,\dots,z_n,y,\lambda)\in\mathbb C^n\times \mathbb C\times\mathfrak h^*$ with values in $V^{\otimes n}$. The transposition $s_i=(i,i+1)$, $i=1,\dots,n-1$, acts as
\begin{gather}\label{e-Si}
f\mapsto S_i(z,y,\lambda)s_i^*f,\qquad S_i(z,y,\lambda)=R(z_i-z_{i+1},y,\lambda-y\sum_{j={i+2}}^{n}h^{(j)})^{(i,i+1)} P^{(i,i+1)},
\end{gather}
where $P\in\operatorname{End}(V\otimes V)$ is the flip $u\otimes v\mapsto v\otimes u$ and $s_i^*$ acts on functions by permutation of $z_i$ with $z_{i+1}$.

To a dynamical $R$-matrix $R$ there corresponds a category ``of representations of the dynamical quantum group associated with $R$''. Fix $y\in\mathbb C$ and let $\mathbb K$ be the field of meromorphic functions of $\lambda\in\mathfrak h^*$ and for $\mu\in\mathfrak h^*$ let $\tau_\mu^*\in\operatorname{Aut}(\mathbb K)$ be the automorphism $\tau_\mu^*f(\lambda)=f(\lambda+y\mu)$. An object of this category is a $\mathbb K$-vector space $W=\oplus_{\mu\in\mathfrak h^*}W_\mu$, which is a semisimple module over $\mathfrak h$, with finite dimensional weight spaces $W_\mu$, together with an endomorphisms $L(w)\in\operatorname{End}_{\mathfrak h}(V\otimes W)$, depending on $w\in U\subset\mathbb C$ for some open dense set $U$, such that
\begin{enumerate}\itemsep=0pt
\item[(i)] $L(w)u\otimes fv=(\mathrm{id}\otimes\tau^*_{-\mu}f) L(w)u\otimes v$, $f\in \mathbb K$, $u\in V_\mu,v\in W$.
\item[(ii)] $L$ obeys the {\it RLL} relations:
 \begin{gather*}
 R\big(w_1-w_2,y,\lambda-y h^{(3)}\big)^{(12)}L(w_1)^{(13)} L(w_2)^{(23)} \\
 \qquad {}= L(w_2)^{(23)}L(w_1)^{(13)} R(w_1-w_2,y,\lambda)^{(12)}.
 \end{gather*}
\end{enumerate}
Morphisms $(W_1,L_{W_1})\to (W_2,L_{W_2})$ are $\mathbb K$-linear maps $\varphi\colon W_1\to W_2$ of $\mathfrak h$-modules, commuting with the action of the generators, in the sense that $L_{W_2}(w)\,\mathrm{id}_V\otimes \varphi=\mathrm{id}_V\otimes \varphi\, L_{W_1}(w)$ for all $w$ in the domain of definition. The dynamical quantum group itself may be defined as generated by Laurent coefficients of matrix elements of $L(w)$ subject to the {\em RLL} relations, see \cite{Konno} for a recent approach in the case of elliptic dynamical quantum groups and for the relations with other definitions of elliptic quantum groups.

The basic example of a representation is the vector evaluation representation $V(z)$ with evaluation point $z\in\mathbb C$. The vector representation has $W=V\otimes_{\mathbb C} \mathbb K$ and
\begin{gather*}
 L(w)v\otimes u=R(w-z,y,\lambda)v\otimes\tau_{-\mu}^* u,\qquad v\in V_\mu,\qquad u\in W.
\end{gather*}
Here $\tau^*_{-\mu} (v\otimes f)=v\otimes \tau_{-\mu}^*f$ for $v\in V$ and
$f\in\mathbb K$, and $R$ acts as a multiplication operator.

More generally we have the tensor product of evaluation
representations $V(z_1)\otimes\cdots\otimes V(z_n)$ with
$ W=V^{\otimes n}\otimes \mathbb K, $ and, by numbering the factors of
$V\otimes V^{\otimes n}$ by $0,1,\dots,n$,
\begin{gather}
 L(w)v\otimes u =R\left(w-z_1,y,\lambda-y\sum_{i=2}^nh^{(i)}\right)^{(01)} R\left(w-z_2,y,\lambda-y \sum_{i=3}^nh^{(i)}\right)^{(02)}\cdots\nonumber\\
\hphantom{L(w)v\otimes u =}{}\times R(w-z_n,y,\lambda)^{(0,n)} v\otimes\tau_{-\mu}^* u, \qquad v\in V_\mu,\qquad u\in W.\label{e-tensor}
\end{gather}
For generic $z_1,\dots, z_n$ the tensor products does not essentially depend on the ordering of the factors: the operators $S_i$ defined above are isomorphisms of representations
\begin{gather*}
 V(z_1)\otimes \cdots \otimes V(z_i)\otimes V(z_{i+1})\otimes\cdots\otimes V(z_n) \to V(z_1)\otimes \cdots \\
 \qquad{}\otimes V(z_{i+1})\otimes V(z_i)\otimes\cdots\otimes V(z_n).
\end{gather*}
\begin{Remark} It is convenient to consider $L$-operators $L(w)$, such as \eqref{e-tensor}, which are meromorphic functions of $w$ and are thus only defined for $w$ in an open dense set. But one may prefer to consider
only representations with $L(w)$ defined for all $w\in\mathbb C$. This may be obtained for the representation given by \eqref{e-tensor} by replacing $L(w)$ by the product of $L(w)$ with $\prod\limits_{a=1}^n\theta(w-z_a+y)$.
\end{Remark}
\subsection{Duality and gauge transformations}\label{ss-2.2}
Suppose that $R(z,y,\lambda)$ is a dynamical $R$-matrix with $\mathfrak h$-module $V$. Let $V^\vee=\oplus_\mu (V^\vee)_\mu$ with weight space $(V^\vee)_\mu$ the dual space to $V_\mu$. Then $R^\vee(z,y,\lambda)=\big(R(z,y,\lambda)^{-1}\big)^*$, the dual map to $R(z,y,\lambda)^{-1}$, is a dynamical $R$-matrix with values in $\operatorname{End}_{\mathfrak h}(V^\vee\otimes V^\vee)$. It is called the dual $R$-matrix to $R$.

Another way to get new $R$-matrices out of old is by a gauge transformation. Let $\psi_V(\lambda)$ be a meromorphic function on $\mathbb C\times \mathfrak h^*$ with values in $\operatorname{Aut}_{\mathfrak h}(V)$. Let
$\psi_{V\otimes V}(\lambda)=\psi_V\big(\lambda-y h^{(2)}\big)^{(1)}\psi_V(\lambda)^{(2)}\in\operatorname{End}_{\mathfrak h}(V\otimes V)$. Then
\begin{gather*}
 R_\psi(z,y,\lambda)=\psi_{V\otimes V}(\lambda)^{-1}R(z,y,\lambda)\psi_{V\otimes V}(\lambda)^{(21)}
\end{gather*}
is another dynamical $R$-matrix. The corresponding representations of the symmetric group are related by the isomorphism
\begin{gather*} \psi_{V^{\otimes n}}(\lambda)=\prod_{i=1}^n\psi_{V}\bigg(\lambda-y\sum_{j=i+1}^nh^{(j)}\bigg)^{(i)}.\end{gather*}

\subsection[The elliptic dynamical quantum group $E_{\tau,y}(\mathfrak{gl}_2)$]{The elliptic dynamical quantum group $\boldsymbol{E_{\tau,y}(\mathfrak{gl}_2)}$}\label{ss-2.3}

In this paper, we focus on the dynamical quantum group $E_{\tau,y}(\mathfrak{gl}_2)$. The corresponding $R$-matrix is the case $N=2$ of Example \ref{example-1}. With respect to the basis $v_1\otimes v_1$, $v_1\otimes v_2$, $v_2\otimes v_1$, $v_2\otimes v_2$,
\begin{gather*}
R(z,y,\lambda)= \left(
 \begin{matrix}
 1 & 0 & 0 & 0\\
 0 & \alpha(z,y,\lambda) & \beta(z,y,\lambda) & 0\\
 0 & \beta(z,y,-\lambda) & \alpha(z,y,-\lambda) & 0\\
 0 & 0 & 0& 1
 \end{matrix} \right),
\end{gather*}
where $\lambda=\lambda_1-\lambda_2$. Since $R$ depends only on the difference $\lambda_1-\lambda_2$ it is convenient to replace~$\mathfrak h$ by the 1-dimensional subspace $\mathbb C$ spanned by $h=\mathrm{diag}(1,-1)$. Then, under the identification $\mathfrak h\cong \mathbb C$ via the basis $h$, $v_1$ has weight 1 and $v_2$ has weight $-1$. Let $(W,L)$ be a representation of $E_{\tau,y}(\mathfrak {gl}_2)$ and write $L(w)=\sum\limits_{i,j=1}^2 E_{ij}\otimes L_{ij}(w)$. Then $L_{ij}(w)$ maps $W_\mu$ to $W_{\mu+2(i-j)}$ and for $f(\lambda)\in \mathbb K$, $L_{i2}(w)f(\lambda) =f(\lambda+y)L_{i2}(w)$ and $L_{i1}(w)f(\lambda) =f(\lambda-y)L_{i1}(w)$.
\begin{Example}[the vector representation $V(z)$] \label{example-2}Let $V=\mathbb C^2$ with basis $v_1$, $v_2$, then
 \begin{gather*}
 L_{11}(w)v_1=v_1,\qquad L_{22}(w)v_2=v_2,\\
 L_{11}(w)v_2= \frac {\theta(w-z)\theta(\lambda+y)} {\theta(w-z-y)\theta(\lambda)} v_2,\qquad
 L_{22}(w)v_1= \frac {\theta(w-z)\theta(\lambda-y)} {\theta(w-z-y)\theta(\lambda)} v_2,\\
 L_{12}(w)v_1 =-\frac{\theta(\lambda+w-z)\theta(y)} {\theta(w-z-y)\theta(\lambda)} v_2, \qquad
L_{21}(w)v_2=-\frac{\theta(\lambda-w+z)\theta(y)} {\theta(w-z-y)\theta(\lambda)} v_1,
 \end{gather*}
and the action on other basis vectors is 0.
\end{Example}

\subsection{The Gelfand--Zetlin subalgebra}
Let $W$ be a representation of the elliptic dynamical quantum group $E_{\tau,y}(\mathfrak{gl}_2)$. Then $L_{22}(w)$, $w\in\mathbb C$ and the quantum determinant \cite{FelderVarchenko}
\begin{gather}\label{e-determinant1}
 \Delta(w)=\frac{\theta(\lambda-y h)}{\theta(\lambda)} (L_{11}(w+y)L_{22}(w)-L_{21}(w+y)L_{12}(w))
\end{gather}
generate a commutative subalgebra of $\operatorname{End}_{\mathfrak h}(W)$. It is called the Gelfand--Zetlin subalgebra.

\section{Shuffle products and weight functions}\label{s-3}
Weight functions are special bases of spaces of sections of line bundles on symmetric powers of elliptic curves. They appear in the theory of hypergeometric integral representation of Knizhnik--Zamolodchikov equations. In~\cite{FelderTarasovVarchenko1997} they were characterized as tensor product bases of a space of function for a suitable notion of tensor products. In this approach the $R$-matrices for highest weight representations of elliptic quantum groups arise as matrices relating bases obtained from taking different orderings of factors in the tensor product. We review and extend the construction of~\cite{FelderTarasovVarchenko1997} in the special case of products of vector representations.

\subsection{Spaces of theta functions}\label{ss-3.1}
\begin{Definition}\label{def-theta-} Let $z\in\mathbb C^n,y\in\mathbb C,\lambda\in\mathbb C$ and define $\Theta^-_k(z,y,\lambda)$ to be the space of entire holomorphic functions $f(t_1,\dots,t_k)$ of $k$ variables such that
 \begin{enumerate}\itemsep=0pt
 \item For all permutations $\sigma\in S_k$, $f(t_{\sigma(1)},\dots,t_{\sigma(k)})=f(t_1,\dots,t_k)$.
 \item For all $r,s\in\mathbb Z$, the meromorphic function
 \begin{gather*}
 g(t_1,\dots,t_k)=\frac{f(t_1,\dots,t_k)}{\prod\limits_{j=1}^k\prod\limits_{a=1}^n \theta(t_j-z_a)}
 \end{gather*}
 obeys
 \begin{gather*}
 g(t_1,\dots, t_i+r+s\tau,\dots,t_k)={\rm e}^{2\pi {\rm i}s(\lambda-ky)}g(t_1,\dots,t_i,\dots,t_k).
 \end{gather*}
 \end{enumerate}
\end{Definition}
\begin{Definition}\label{def-theta+} Let $z\in\mathbb C^n$, $y\in\mathbb C$, $\lambda\in\mathbb C$ and define $\Theta^+_k(z,y,\lambda)$ to be the space of entire holomorphic functions $f(t_1,\dots,t_k)$ of $k$ variables such that
 \begin{enumerate}\itemsep=0pt
 \item For all permutations $\sigma\in S_k$, $f(t_{\sigma(1)},\dots,t_{\sigma(k)})=f(t_1,\dots,t_k)$.
 \item For all $r,s\in\mathbb Z$, the meromorphic function
 \begin{gather*}
 g(t_1,\dots,t_k)=\frac{f(t_1,\dots,t_k)}{\prod\limits_{j=1}^k\prod\limits_{a=1}^n \theta(t_j-z_a+y)},
 \end{gather*}
 obeys
 \begin{gather*}
 g(t_1,\dots, t_i+r+s\tau,\dots,t_k)={\rm e}^{-2\pi{\rm i}s(\lambda-ky)}g(t_1,\dots,t_i,\dots,t_k).
 \end{gather*}
 \end{enumerate}
\end{Definition}
\begin{Remark}\label{remark-1} These spaces are spaces of symmetric theta functions of degree $n$ in $k$ variables and have dimension ${n+k-1 \choose k}$. Actually $\Theta^-$ depends on the parameters only through the combination $\sum\limits_{a=1}^n z_a+\lambda-ky$ and $\Theta^+$ through the combination $\sum\limits_{a=1}^nz_a-\lambda-(n-k)y$.
\end{Remark}
\begin{Example}\label{example-4} For $z\in\mathbb C$ and all $k=0,1,2,\dots$, $\Theta^-_k(z,y,\lambda)$ is a one-dimensional vector space
 spanned by
 \begin{gather*}
 \omega^-_k(t;z,y,\lambda)=\prod_{j=1}^k\theta(\lambda-t_j+z-ky),
 \end{gather*}
 $\Theta^+_k(z,y,\lambda)$ is a one-dimensional vector space spanned by
 \begin{gather*}
 \omega^+_k(t;z,y,\lambda) =\prod_{j=1}^k\theta(\lambda+t_j-z+(1-k)y).
 \end{gather*}
\end{Example}
 \begin{Remark}\label{remark-2} For $z\in\mathbb C^n$, $y,\lambda\in\mathbb C$, $\Theta^-_k(z,y,\lambda)=\Theta^+_k(z,y,-\lambda-(n-2k)y)$ and $\omega^+_k(t;z,y,\lambda)=(-1)^k\omega^-_k(t;z,y,-\lambda-(1-2k)y)$. It is however better to keep the two spaces distinct as they will be given a different structure.
 \end{Remark}
\subsection{Shuffle products}\label{ss-3.2}
Let Sym denote the map sending a function $f(t_1,\dots,t_k)$ of $k$ variables to the symmetric function $\sum\limits_{\sigma\in S_n}f(t_{\sigma(1)},\dots,t_{\sigma(k)})$.
\begin{Proposition}\label{prop-0} Let $n=n'+n''$, $k=k'+k''$ be non-negative integers, $z\in\mathbb C^n$, $z'=(z_1,\dots,z_{n'})$, $z''=(z_{n'+1},\dots, z_n)$. Then the shuffle product
 \begin{gather*}
 *\colon \ \Theta^\pm_{k'}(z',y,\lambda+y(n''-2k''))\otimes \Theta^\pm_{k''}(z'',y,\lambda)\to\Theta^\pm_k(z,y,\lambda),
 \end{gather*}
 sending $f\otimes g$ to
 \begin{gather*}
 f*g(t)=\frac1{k'!k''!} \operatorname{Sym}\big(f(t_1,\dots,t_{k'}) g(t_{k'+1},\dots,t_k)\varphi^\pm(t,z,y)\big),
 \end{gather*}
 with
 \begin{align*}
 \varphi^-(t,z,y)&=\prod_{j=1}^{k'} \prod_{l=k'+1}^k \frac{\theta(t_l-t_j+y)} {\theta(t_l-t_j)} \prod_{l=k'+1}^k\prod_{a=1}^{n'}\theta(t_l-z_a) \prod_{j=1}^{k'}\prod_{b=n'+1}^n\theta(t_j-z_b+y), \\
 \varphi^+(t,z,y)&=\prod_{j=1}^{k'} \prod_{l=k'+1}^k \frac{\theta(t_j-t_l+y)}{\theta(t_j-t_l)} \prod_{l=k'+1}^k\prod_{a=1}^{n'}\theta(t_l-z_a+y) \prod_{j=1}^{k'}\prod_{b=n'+1}^n\theta(t_j-z_b),
 \end{align*}
 is well-defined and associative, in the sense that $(f*g)*h=f*(g*h)$, whenever defined.
\end{Proposition}
\begin{Remark}\label{remark-3} In the formula for $f*g$ in Proposition~\ref{prop-0} we can omit the factor $1/k'!k''!$ and replace the sum over permutations defining Sym by the sum over $(k',k'')$-shuffles, namely permutations $\sigma\in S_{k}$ such that $\sigma(1)<\dots<\sigma(k')$ and $\sigma(k'+1)<\cdots<\sigma(k)$.
\end{Remark}
\begin{proof} This is essentially the first part of Proposition~3 of \cite{FelderTarasovVarchenko1997} in the special case of weights $\Lambda_i=1$. The proof is straightforward: the apparent poles at $t_j=t_l$ are cancelled after the symmetrization since $\theta(t_j-t_l)$ is odd under interchange of $t_j$ with $t_l$. Thus $f*g$ is a symmetric entire function. One then checks that every term in the sum over permutations has the correct transformation property under lattice shifts.
\end{proof}
\begin{Proposition}\label{prop-1} The maps $*$ of Proposition~{\rm \ref{prop-0}} define isomorphisms
 \begin{gather*}
 \oplus_{k'=0}^k \Theta^\pm_{k'}(z',y,\lambda+y(n''-2k'')) \otimes \Theta^\pm_{k-k'}(z'',y,\lambda) \to \Theta^\pm_k(z,y,\lambda)
 \end{gather*}
 for generic $z$, $y$, $\lambda$.
\end{Proposition}
We prove this Proposition in \ref{ss-3.12} below.

\subsection{Vanishing condition}\label{ss-3.3}
The shuffle product $*$ preserves subspaces defined by a vanishing condition. It is the case of the fundamental weight of a condition introduced in~\cite[Section~8]{FelderTarasovVarchenko1997} for general integral dominant weights.

Let $(z,y,\lambda)\in\mathbb C^n\times \mathbb C\times \mathbb C$. We define a subspace $\bar\Theta^\pm_k(z,y,\lambda)\subset\Theta^\pm_k(z,y,\lambda)$ by a~vanishing condition:
\begin{gather*}
 \bar\Theta^\pm_k(z,y,\lambda) = \begin{cases} \Theta^\pm_k(z,y,\lambda)&\text{if $k=0,1$,}\\
\{f \colon f(t_1,\dots,t_{k-2},z_a,z_a-y)=0,\,1\leq a\leq n,\, t_i\in\mathbb C\} &\text{if $k\geq2$.}
\end{cases}
\end{gather*}
\begin{Example}\label{example-5} For $n=1$,
 \begin{gather*}
 \bar\Theta^\pm_k(z,y,\lambda) =\begin{cases}
 \Theta^\pm_k(z,y,\lambda)\cong \mathbb C,&\text{$k=0,1$,}\\
 0,&\text{$k\geq2$.}
 \end{cases}
 \end{gather*}
 Indeed, the condition is vacuous if $k\leq 1$ and if $k\geq2$ then $\omega^\pm_k(z;z-y,t_3,\dots)=\theta(\lambda-ky)\theta(\lambda+(1-k)y)$ times a nonzero function. For $k=1$, $n\geq1$, $\bar\Theta^\pm_1(z,y,\lambda) = \Theta^\pm_1(z,y,\lambda)$. For $k=2$, $n=2$, $\bar\Theta^\pm_2(z_1,z_2,y,\lambda)$ is a one-dimensional subspace of the three-dimensional space $\Theta^\pm_2(z,y,\lambda)$.
\end{Example}
\begin{Proposition}\label{prop-2} The shuffle product restricts to a map
\begin{gather*}
 \oplus_{k'=0}^k \bar\Theta^\pm_{k'} (z',y,\lambda+y(n''-2k'')) \otimes \bar\Theta^\pm_{k-k'} (z'',y,\lambda) \to \bar\Theta^\pm_k(z,y,\lambda),
 \end{gather*}
 which is an isomorphism for generic values of the parameters.
\end{Proposition}
The proof is postponed to Section~\ref{ss-3.12} below. By iteration we obtain shuffle multiplication maps
\begin{gather*}
 \bar\Phi_k^{\pm}(z,y,\lambda)\colon\bigoplus_{\Sigma k_a=k} \bigotimes_{a=1}^n \bar\Theta^\pm_{k_a}\bigg(z_a,y,\lambda-y\sum_{b=a+1}^n(2k_a-1)\bigg) \to \bar\Theta^\pm_k(z_1,\dots,z_n,y,\lambda),
\end{gather*}
defined for $(z,y,\lambda)\in\mathbb C^n\times \mathbb C\times \mathbb C$ and $k=0,1,2,\dots$. The direct sum is over the ${n \choose k}$ $n$-tuples $(k_1,\dots,k_n)$ with sum $k$ and $k_a\in\{0,1\}$, $a=1,\dots,n$.
\begin{Corollary}\label{cor-0} The maps $\bar\Phi_k^{\pm}(z,y,\lambda)$ are isomorphisms for generic $(z,y,\lambda)\in\mathbb C^n\times \mathbb C\times \mathbb C$.
\end{Corollary}

Thus, for generic $z,y,\lambda\in\mathbb C^n\times\mathbb C\times\mathbb C$, $\bar\Theta^\pm_k(z,y,\lambda)$ has dimension ${n \choose k}$ and is zero if $k>n$.

\subsection{Duality}\label{ss-3.4}
\begin{Proposition}\label{p-duality} The identification
\begin{gather*}
 \varrho\colon \ \Theta^-_k(z,y,\lambda)\to \Theta^+_k(z,y,-\lambda-(n-2k)y)
\end{gather*}
of Remark~{\rm \ref{remark-2}} $($the identity map$)$ restricts to an isomorphism
\begin{gather*}
 \bar \Theta^-_k(z,y,\lambda)\to \bar\Theta^+_k(z,y,-\lambda-(n-2k)y),
\end{gather*}
also denoted by $\varrho$. For $f\in\Theta^-_{k'}$ and $g\in\Theta^-_{k''}$ as in Proposition~{\rm \ref{prop-0}}, the shuffle product $\varrho(g)*\varrho(f)$ is well-defined and obeys
\begin{gather*}
 \varrho(f*g)=\varrho(g)*\varrho(f).
\end{gather*}
\end{Proposition}
\begin{proof} It is clear that the vanishing condition is preserved. The last claim follows from the identity
 \begin{gather*}
 \varphi^-(t,z,y)= \varphi^+(t_{k'+1},\dots,t_k,t_1,\dots,t_{k'},z_{n'+1},\dots,z_n,z_1,\dots,z_{n'},y)
 \end{gather*}
 for the functions appearing in the definition of the shuffle product.
\end{proof}

\begin{Remark}\label{remark-4} For $n=1$ we have $\varrho(\omega^-_k)=(-1)^k\omega^+_k$, see Example \ref{example-4}.
\end{Remark}
\subsection{Weight functions}\label{ss-3.5}
For $(z,y,\lambda)\in\mathbb C^n\times\mathbb C\times\mathbb C$, let
\begin{gather*}
\bar\Theta^\pm(z,y,\lambda)=\oplus_{k=0}^n \bar\Theta^\pm_k(z,y,\lambda).
\end{gather*}
It is an $\mathfrak h$-module with $\bar\Theta^\pm_{k}$ of weight $-n+2k$. Let $v_1$, $v_2$ be the standard basis of $\mathbb C^2$. If $n=1$, we identify $\bar \Theta^\pm(z,y,\lambda)$ with $\mathbb C^2$ via the map
$\omega^{\pm}_1\mapsto v_1$, $\omega^{\pm}_0\mapsto v_2$. Then $\bar\Phi^\pm(z,y,\lambda)=\oplus_k\bar\Phi^\pm_k(z,y,\lambda)$ is a linear map
\begin{gather*}
 \big(\mathbb C^2\big)^{\otimes n}\to \bar\Theta^\pm(z,y,\lambda).
\end{gather*}
It is a homomorphism of $\mathfrak h$-modules. Then a basis of $\big(\mathbb C^2\big)^{\otimes n}$ is labeled by subsets $I$ of $[n]=\{1,\dots,n\}$: $v_I=v_{j(1)}\otimes\cdots\otimes v_{j(n)}$ with $j(a)=2$ if $a\in I$ and $j(a)=1$ if $a\in\bar I$, the complement of~$I$.
\begin{Definition} The {\em weight functions} $\omega^\pm_I(t;z,y,\lambda)$ are the functions
 \begin{gather*}
 \omega^\pm_I(\cdot;z,y,\lambda)=\bar\Phi^{\pm}(z,y,\lambda)v_I\in \bar \Theta^\pm(z,y,\lambda).
 \end{gather*}
\end{Definition}
In particular, for $n=1$, $\omega^\pm_\varnothing=\omega^\pm_0$, $\omega^\pm_{\{1\}}=\omega_1^\pm$. Corollary~\ref{cor-0} implies:
\begin{Proposition}\label{p-basis} Let $(z,y,\lambda)$ be generic. The weight functions $\omega^\pm_I(\cdot; z,y,\lambda)$ with $I\subset [n]$, $|I|=k$ form a basis of the space $\bar\Theta^\pm_k(z,y,\lambda)$ of theta
functions obeying the vanishing condition.
\end{Proposition}
\begin{Example}\label{example-6} For $k=1$ and $n=1,2\dots$, $z\in\mathbb C^n$, $y\in\mathbb C$, $\lambda\in\mathbb C$, $a=1,\dots,n$,
 \begin{gather*}
 \omega^-_{\{a\}}(t;z,y,\lambda) =\theta(\lambda-t+z_a+y(n-a-1)) \prod_{b=1}^{a-1}\theta(t-z_b) \prod_{b=a+1}^n \theta(t-z_b+y),\\
 \omega^+_{\{a\}}(t;z,y,\lambda) =\theta(\lambda+t-z_a+y(n-a)) \prod_{b=1}^{a-1}\theta(t-z_b+y) \prod_{b=a+1}^n \theta(t-z_b).
 \end{gather*}
\end{Example}
\subsection[$R$-matrices]{$\boldsymbol{R}$-matrices}\label{ss-3.6}
Note that while $\bar\Theta^{\pm}_k(z,y,\lambda)$ is independent of the ordering of $z_1,\dots,z_n$ the map $\bar\Phi_k^\pm$ does depend on it and different orderings are related by $R$-matrices, as we now describe. We
define $R$-matrices $R^\pm(z,y,\lambda)\in\operatorname{End}_{\mathfrak h}\big(\mathbb C^2\otimes\mathbb C^2\big)$ by
\begin{gather*}
 R^{\pm}(z_1-z_2,y,\lambda) =\bar\Phi^\pm(z_1,z_2,y,\lambda)^{-1} \bar\Phi^\pm(z_2,z_1,y,\lambda)P,
\end{gather*}
where $Pu\otimes v=v\otimes u$ is the flip of factors. Up to duality and gauge transformation, these $R$-matrices coincide with the elliptic $R$-matrix of Section~\ref{ss-2.3}:
\begin{Proposition}\label{prop-3}\quad
\begin{enumerate}\itemsep=0pt
 \item[$(i)$] Let $s_i\in S_n$ be the transposition $(i,i+1)$. Then
 \begin{gather*}
 \bar\Phi^\pm(s_iz,y, \lambda) =\bar\Phi^\pm(z,y,\lambda) R^\pm\bigg(z_i-z_{i+1},y,\lambda -y\sum_{j=i+2}^nh^{(j)}\bigg)^{(i,i+1)}P^{(i,i+1)}.
 \end{gather*}
 \item[$(ii)$] The $R$-matrices $R^\pm$ obey the dynamical Yang--Baxter equation \eqref{e-YBE} and the inversion relation~\eqref{e-inversion}.
 \item[$(iii)$] With respect to the basis $v_1\otimes v_1$, $v_1\otimes v_2$, $v_2\otimes v_1$, $v_2\otimes v_2$ of $\mathbb C^2\otimes \mathbb C^2$,
 \begin{gather*}
 R^-(z,y,\lambda)= \left(
 \begin{matrix}
 1 & 0 & 0 & 0\\
 0 & \alpha(-z,y,-\lambda) & \beta(-z,y,\lambda) & 0\\
 0 & \beta(-z,y,-\lambda) & \alpha(-z,y,\lambda) & 0\\
 0 & 0 & 0& 1
 \end{matrix}
 \right)=R^\vee(z,y,\lambda)
 \end{gather*}
 is the dual $R$-matrix, see Section~{\rm \ref{ss-2.2}}, with the standard identification of $\mathbb C^2$ with $\big(\mathbb C^2\big)^*$ and
 \begin{gather*}
 R^+(z,y,\lambda)= \left(
 \begin{matrix}
 1 & 0 & 0 & 0\\
 0 & \alpha(z,y,-\lambda) & \beta(z,y,\lambda) & 0\\
 0 & \beta(z,y,-\lambda) & \alpha(z,y,\lambda) & 0\\
 0 & 0 & 0& 1
 \end{matrix}
 \right)= R_\psi(z,y,\lambda)
 \end{gather*}
 is the gauge transformed $R$-matrix with
 \begin{gather*}
 \psi(\lambda)=\left(
 \begin{matrix}
 \theta(\lambda)\theta(\lambda-y) & 0 \\
 0 & 1
 \end{matrix}
 \right).
 \end{gather*}
 \end{enumerate}
\end{Proposition}

\begin{Corollary}\label{cor-1} Let $(z,y,\lambda)\in\mathbb C^n\times \mathbb C\times\mathbb C$ be generic and set
 \begin{gather*}
 S_i(z,y,\lambda)= R\bigg(z_i-z_{i+1},y,\lambda-y\sum_{j={i+2}}^{n}h^{(j)}\bigg)^{(i,i+1)}P^{(i,i+1)} \in\operatorname{End}_{\mathfrak h}\big(\big(\mathbb C^2\big)^{\otimes n}\big),
 \end{gather*}
 $i=1,\dots,n-1$, cf.~\eqref{e-Si}.
 \begin{enumerate}\itemsep=0pt
 \item[$(i)$] For $t\in\mathbb C^k$, let $\omega^-(t;z,y,\lambda)=\sum\limits_{I\subset[n], |I|=k}\omega^-_I(t;z,y,\lambda)v_I$. Then
 \begin{gather*}
 \omega^-(t;z,y,\lambda)=S_i(z,y,\lambda) \omega^-(t;s_iz,y,\lambda).
 \end{gather*}
 \item[$(ii)$] Let $\psi_{V^{\otimes n}}(\lambda)=\prod\limits_{i=1}^n\psi\big(\lambda-y \sum\limits_{j>i}h^{(j)}\big)^{(i)}$, cf.~Section~{\rm \ref{ss-2.2}}. Then
 \begin{gather*}
 \bar\Phi^+(s_iz,y,\lambda)\psi_{V^{\otimes n}}(\lambda)^{-1} =\bar\Phi^+(z,y,\lambda)\psi_{V^{\otimes n}}(\lambda)^{-1} S_i(z,y,\lambda).
 \end{gather*}
 \end{enumerate}
\end{Corollary}

\subsection{A geometric representation}\label{ss-3.7}
Let $z_1,\dots,z_n,y,\lambda$ be generic and $w\in\mathbb C$. Recall that we identify $\bar\Theta^+(w,y,\lambda)$ with $V=\mathbb C^2$ via the basis $\omega^+_1$, $\omega^+_0$. Consider the shuffle products\footnote{The compressed notation we are using might be confusing: the map $p_+$ is actually defined on $\oplus_k \bar\Theta^+(w,y,\lambda+(n-2k)y)\otimes\bar\Theta^+_k(z,y,\lambda)$. The identification of the first factor with $V$ depends on $k$ through the $\lambda$-dependence of the basis vectors $\omega^+_i$.}
\begin{gather*}
p_+\colon \ V \otimes \bar\Theta^+(z_1,\dots,z_n,y,\lambda)\to\bar\Theta^+(w,z_1,\dots,z_n,y,\lambda),\\
p_-\colon \ \bar\Theta^+\big(z_1,\dots,z_n,y,\lambda-yh^{(2)}\big)\otimes V \to \bar\Theta^+(w,z_1,\dots,z_n,y,\lambda).
\end{gather*}
Then varying $w$ and denoting $P$ the flip of tensor factors, we get a homomorphism
\begin{gather*}
 \ell(w,y,\lambda)=p_+^{-1}\circ p_-\circ P\in\operatorname{Hom}\big(V\otimes \bar\Theta^+\big(z,y,\lambda-y h^{(1)}\big),V\otimes \bar\Theta^+(z,y,\lambda)\big).
\end{gather*}
By construction it obeys the dynamical Yang--Baxter equation
\begin{gather}
 R^+\big(w_1-w_2,y,\lambda-y h^{(3)}\big)^{(12)}\ell(w_1,y,\lambda)^{(13)} \ell\big(w_2,y,\lambda-y h^{(1)}\big)^{(23)}\notag\\
\qquad{} = \ell(w_1,y,\lambda)^{(23)}\ell\big(w_2,y,\lambda-y h^{(2)}\big)^{(13)} R^+\big(w_1-w_2,y,\lambda-y h^{(3)}\big)^{(12)}\label{e-YBE1}
\end{gather}
in $\operatorname{Hom}\big(V\otimes V\otimes \bar\Theta^+\big(z,y,\lambda-y\big(h^{(1)}+h^{(2)}\big)\big), V\otimes V\otimes \bar\Theta^+(z,y,\lambda)\big)$. By varying $\lambda$ we obtain a~representation of the elliptic dynamical quantum group as follows. Let $(z,y)\in\mathbb C^n\times \mathbb C$ be generic and consider the space $\bar\Theta^+_k(z,y)_{\mathrm{reg}}$ of holomorphic functions $f(t,\lambda)$ on $\mathbb C^k\times \mathbb C$ such that for each fixed $\lambda$, $t\mapsto f(t,\lambda)$ belongs to $\bar\Theta^+(z,y,\lambda)$. It is a module over the ring $\mathcal O(\mathbb C)$ of holomorphic functions of $\lambda$. We set
\begin{gather*}
 \bar\Theta^+_k(z,y)=\bar\Theta^+_k(z,y)_{\mathrm{reg}}\otimes_{\mathcal O(\mathbb C)}\mathbb K.
\end{gather*}
It is a finite dimensional vector space over $\mathbb K$, and for generic $z$, $y$ it has a basis given by weight functions~$\omega^+_I$, $|I|=k$.

\begin{Proposition}\label{prop-7} Let $z_1,\dots,z_n,y$ be generic complex numbers. Then
 \begin{gather*}
 \bar\Theta^+(z_1,\dots,z_n,y)=\oplus_{k=0}^n\bar\Theta^+_k(z_1,\dots,z_n,y).
 \end{gather*}
 is a representation of the elliptic quantum group $E_{\tau,y}(\mathfrak{gl}_2)$ with the $L$-operator
 \begin{gather*}
 L(w)(v\otimes u)=\psi\big(\lambda-yh^{(2)}\big)^{(1)}\ell(w,y,\lambda)\big(\psi(\lambda)^{-1}\big)^{(1)}(v\otimes\tau_{-\mu}^*u),\qquad v\in V_\mu.
 \end{gather*}
 Here $\psi$ is the gauge transformation of Proposition~{\rm \ref{prop-3}}.
\end{Proposition}
\begin{proof}The homomorphisms $\ell$ obey the {\em RLL}-type relations \eqref{e-YBE1} with $R$-matrix $R^+$ which, according to Proposition \ref{prop-3}, is obtained from $R$ by the gauge transformation $\psi$. It is easy to check that
\begin{gather*}
 \hat\ell(w,y,\lambda)= \psi\big(\lambda-yh^{(2)}\big)^{(1)}\ell(w,y,\lambda)\big(\psi(\lambda)^{-1}\big)^{(1)}
\end{gather*}
obeys the same relations but with $R^+$ replaced by $R$. It follows that the corresponding difference operators define a~representation of the elliptic dynamical quantum group.
\end{proof}

\begin{Remark} It follows from the previous section that this representation is isomorphic to the tensor product $V(z_{\sigma(1)})\otimes\cdots\otimes V(z_{\sigma(n)})$ for any permutation $\sigma\in S_n$. However this identification with a~tensor product of evaluation vector representations depends on a~choice of ordering of the $z_1,\dots,z_n$, while $\bar\Theta^+(z_1,\dots,z_n,y,\lambda)$ depends as a representation only on the set $\{z_1,\dots, z_n\}$.
\end{Remark}

\subsection{Pairing}
We define a pairing of $\Theta^-_k$ with $\Theta^+_k$, taken essentially from \cite[Appendix~C]{TV2}. Note that the product of a~function in $\Theta^-_k(z,y,\lambda)$ and a function in $\Theta^+_k(z,y,\lambda)$ divided by the products of Jacobi theta functions in part~(2) of Definitions~\ref{def-theta-} and~\ref{def-theta+}, is a function which is doubly periodic in each variable $t_i$ with poles at $t_i=z_a$ and at $t_i=z_a-y$, $a=1,\dots,n$. It can thus be viewed as a meromorphic function on the Cartesian power $E^k$ of the elliptic curve $E=\mathbb C/(\mathbb Z+\tau\mathbb Z)$.
\begin{Definition} Let $z_1\dots,z_n,y\in E$ such that $z_a\neq z_b+jy$ for all $1\leq a,b\leq n$, $1\leq j\leq n-1$, and let $\gamma\in H_1(E\smallsetminus\{z_1,\dots,z_n\})$ be the sum of small circles around $z_a$, $a=1,\dots,n,$ oriented in counterclockwise direction. Let $D\subset E^k$ be the effective divisor $D=\cup_{a=1}^n\cup_{i=1}^k\big(\big\{t\in E^k\colon t_i=z_a\big\}\cup\big\{t\in E^k\colon t_i=z_a-y\big\}\big)$. The symmetric group $S_k$ acts by permutations on the sections of the sheaf $\mathcal O(D)$ of functions on $E^k$ with divisor of poles bounded by $D$. Let $\langle \ \rangle\colon\Gamma\big(E^k,\mathcal O(D)\big)^{S_k}\to \mathbb C$ be the linear form
 \begin{gather*}
 f\to\langle f\rangle= \frac{\theta(y)^k}{(2\pi{\rm i})^kk!} \int_{\gamma^k} f(t_1,\dots,t_k)\prod_{1\leq i\neq j\leq k} \frac{\theta(t_i-t_j)} {\theta(t_i-t_j+y)}{\rm d}t_1\cdots {\rm d}t_k.
 \end{gather*}
For $k=0$ we define $\langle \ \rangle\colon\mathbb C\to\mathbb C$ to be the identity map.
\end{Definition}
\begin{Lemma}\label{lemma-2} Let $f\in\Gamma\big(E^k,\mathcal O(D)\big)^{S_k}$. Then
 \begin{gather*}
 \langle f\rangle=\theta(y)^k\sum_{1\leq i_1<\cdots<i_k\leq n}\mathrm{res}_{t_1=z_{i_1}}\cdots\mathrm{res}_{t_k=z_{i_k}} \bigg(f(t_1,\dots,t_k)\prod_{i\neq j} \frac{\theta(t_i-t_j)} {\theta(t_i-t_j+y)}\bigg).
 \end{gather*}
\end{Lemma}
\begin{proof} By the residue theorem, $\langle f\rangle$ is a sum of iterated residues at $t_i=z_{a(i)}$ labeled by maps $a\colon [k]\to[n]$. Since $\theta(t_i-t_j)$ vanishes for $t_i=t_j$, only injective maps $a$ contribute non-trivially. Moreover, since the integrand is symmetric under permutations of the variables $t_i$, maps $a$ differing by a~permutation of $\{1,\dots,k\}$ give the same contribution. Thus we can restrict the sum to strictly increasing maps $a$ and cancel the factorial $k!$ appearing in the definition.
\end{proof}
\begin{Definition} Denote $Q=\prod\limits_{i=1}^k\prod\limits_{a=1}^n\theta(t_i-z_a)\theta(t_i-z_a+y)$ and let
 \begin{gather*}
 \langle \ , \ \rangle\colon \ \Theta^-_k(z,y,\lambda)\otimes\Theta^+_k(z,y,\lambda)\to
 \mathbb C
 \end{gather*}
 be the bilinear pairing $ \langle f,g\rangle= \langle fg/Q\rangle $, defined for generic $z\in\mathbb C^n$, $y\in\mathbb C$. Note that $fg/Q$ is an elliptic function of $t_i$ for all~$i$.
\end{Definition}
Here is the explicit formula for the pairing:
\begin{gather}\label{e-pairing}
\langle f,g\rangle=\frac{\theta(y)^k}{(2\pi{\rm i})^kk!}\int_{\gamma^k}\frac{f(t_1,\dots,t_k)g(t_1,\dots,t_k)}{\prod\limits_{i,a}\theta(t_i-z_a)\theta(t_i-z_a+y)}\prod_{i\neq j}\frac{\theta(t_i-t_j)}{\theta(t_i-t_j+y)}
{\rm d}t_1\cdots {\rm d}t_k.
\end{gather}

\begin{Lemma}\label{lemma-3} Let $n=1$. Then $\langle \omega^-_0,\omega^+_0\rangle=1$,
 \begin{gather*}
 \langle \omega^-_1,\omega^+_1\rangle=\theta(\lambda-y) \theta(\lambda),
 \end{gather*}
 and $\langle\omega^-_k,\omega^+_k\rangle=0$ for $k>1$.
\end{Lemma}
\begin{proof} The first claim holds by definition. We have
 \begin{gather*}
 \langle\omega^-_1,\omega^+_1\rangle =\theta(y)\,\mathrm{res}_{t=z}\frac{\theta(\lambda-t+z-y) \theta(\lambda+t-z)}{\theta(t-z)\theta(t-z+y)}{\rm d}t =\theta(\lambda-y) \theta(\lambda).
 \end{gather*}
 For $k\geq 2$, the residue at $t_1=z$ is regular at $t_i=z$ for $i\geq2$ and thus the iterated residue vanishes.
\end{proof}
\begin{Proposition}\label{prop-4}\quad \begin{enumerate}\itemsep=0pt
 \item[$(i)$] The pairing restricts to a non-degenerate pairing $\bar\Theta^-_k(z,y,\lambda) \otimes\bar\Theta^+_k(z,y,\lambda) \to \mathbb C$ for generic $z_1,\dots,z_n,y,\lambda$.
 \item[$(ii)$] In the notation of Proposition~{\rm \ref{prop-2}}, suppose $f_i\!\in\! \bar\Theta^-_{k'_i}(z',y,\lambda+y(n''-2k_i''))$, $g_i\!\in\! \bar\Theta^+_{k''_i}(z'',y,\lambda)$, $i=1,2$ and $k_1'+k_2'=k_1''+k_2''$. Then
 \begin{gather*}
 \langle f_1*f_2,g_1*g_2\rangle= \begin{cases}\langle f_1,g_1\rangle\langle f_2,g_2\rangle, & \text{if $k'_1=k''_1$ and $k'_2=k''_2$}, \\
 0,&\text{otherwise}.
 \end{cases}
 \end{gather*}
 \end{enumerate}
\end{Proposition}
\begin{proof} It is sufficient to prove (ii), since with Lemma~\ref{lemma-3} it implies that, with a proper normalization, weight functions form dual bases with respect to the pairing.

We use Lemma \ref{lemma-2} to compute $\langle f_1*f_2,g_1*g_2\rangle$. Let us focus on the summand in Lemma~\ref{lemma-2} labeled by $i_1<\cdots<i_n$ and suppose $i_s\leq n'<i_{s+1}$. Due to the factor $\theta(t_l-z_a)$ in~$\varphi^-$, see Proposition~\ref{prop-0}, the only terms in the sum over shuffles having nonzero first $s$ residues $\mathrm{res}_{t_1=z_{i_1}}$, \dots, $\mathrm{res}_{t_s=z_{i_s}}$ are those for which $t_1,\dots, t_s$ are arguments of $f_1$. In particular the summand vanishes unless $s\leq k_1'$. Similarly the factors $\theta(t_j-z_b)$ in $\varphi^+$ restrict the sum over shuffles to those terms for which $t_{s+1},\dots,t_k$ are arguments of $g_2$, so that the summand vanishes unless $s\geq k-k_2''=k_2'$. It follows that if $k_1'<k_2'$ then $\langle f_1*f_2,g_1*g_2\rangle$ vanishes and that if $k_1'=k_1''$, the pairing can be computed explicitly as sum over $i_1<\cdots<i_s\leq n'<i_{s+1}<\cdots<i_{k}$, with $s=k_1'$, of terms involving $f_1g_1(z_{i_1},\dots,z_{i_{k_1'}})f_2g_2(z_{i_{k_1'}+1},\dots,z_{i_k})$. The coefficients combine to give $\langle f_1,g_1\rangle\langle f_2,g_2\rangle$.

There remains to prove that the pairing vanishes also if $k_1'>k_2'$. Here is where the vanishing condition comes in. We first consider the case where $k_1'-k_2'=1$ and then reduce the general case to this case.

As above the presence of the vanishing factors in $\varphi^\pm$ imply that the non vanishing residues in Lemma~\ref{lemma-2} are those labeled by $i_1<\cdots<i_k$ such that at least $k_1''$ indices are $\geq n'$ and the corresponding variables $t_i$ are arguments of $g_1$ and at least $k_2'$ indices are $\leq n'$ and the corresponding variables are arguments of $f_2$. If $k_1'-k_2'=1$ there is one variable left and we can write the pairing as a sum of one-dimensional integrals over this variable:
 \begin{gather}\label{e-t-integral}
 I_{A,B}=\int_\gamma \frac{f_1(z_A,t)g_1(z_B)f_2(z_A)g_2(t,z_B)}{h(z_1,\dots,z_n,y,t)}{\rm d}t.
 \end{gather}
Here $z_A=z_{a_1},\dots,z_{a_{k_2'}}$ with $a_i\leq n'$ and $z_B=z_{b_1},\dots,z_{b_{k_1''}}$ with $b_i>n'$. The point is that in $h(z,t)$ several factor cancel and one obtains
\begin{gather*}
 h(z,y,t)=C(z,y)\prod_{c\in \overline {A\cup B}}\theta(t-z_c) \prod_{c\in A\cup B}\theta(t-z_c+y),
 \end{gather*}
for some $t$-independent function $C(z,y)$. Because of the vanishing condition, the integrand in~\eqref{e-t-integral} is actually regular at $t=z_c-y$ and the only poles are at $t=z_c$, $c\in\overline {A\cup B}$. By the residue theorem $I_{A,B}=0$.

Finally, let us reduce the general case to the case where $k_1'-k_2'=1$. We use induction on~$n$. By Lemma~\ref{lemma-3} the pairing vanishes unless $k=1,0$ so there is nothing to prove in this case. Assume that the claim is proved for $n-1$. By Proposition~\ref{prop-2}, we can write $g_1=h_1*m_1$ and $g_2=h_2*m_2$ with $m_i\in\bar\Theta^-_{r_i}(z_{n},y,\lambda)$. By Lemma \ref{lemma-3} we can assume that $r_i\in\{0,1\}$. By the associativity of the shuffle product we can use the result for $k_1'-k_2'=1$ to obtain that the pairing vanishes unless $r_1=r_2$ and
\begin{gather*}
 \langle f_1*g_1,f_2*g_2\rangle=\langle f_1*h_1,f_2*h_2\rangle\langle m_1,m_2\rangle.
 \end{gather*}
By the induction hypothesis, this vanishes unless $k_1'=k_2'$.
\end{proof}

We obtain orthogonality relations for weight functions. To formulate them we introduce some notation. For $I\subset [n]$ and $1\leq j\leq n$ we set
\begin{gather}
 n(j,I)=|\{l\in[n]\,|\,l\in I, l>j\}|,\nonumber\\
 w(j,I)=n(j,I)-n(j,\bar I).\label{eq-w}
\end{gather}
Thus $-w(j,I)$ the sum of the weights of the tensor factors to the right of the $j$-th factor in $v_I$.
 \begin{Corollary}[{cf.~\cite[Theorem~C.4]{TV2}}]\label{cor-2}
 \begin{gather*}
 \langle\omega^-_I,\omega^+_J\rangle = \delta_{I,J}\psi_I(y,\lambda),
 \end{gather*}
 where $\psi_I(y,\lambda)= \prod\limits_{j\in I}\theta(\lambda-(w(j,I)+1)y)\theta(\lambda-w(j,I)y)$.
 \end{Corollary}

 \subsection{Normalized weight functions}\label{ss-3.9}
 By construction the weight functions $\omega^\pm_I$ are entire functions of all variables and obey the vanishing conditions
 \begin{gather*}
 \omega^\pm_I(z_a,z_a-y,t_3,\dots,t_k;z,y,\lambda)=0, \qquad a=1,\dots,n.
 \end{gather*}
 This motivates the following definition.
 \begin{Definition}
 The {\em normalized weight functions} $w_I^\pm$ are the functions
 \begin{gather*}
 w_I^-(t;z,y,\lambda) =\frac{\omega^-_I(t;z,y,\lambda)}{\prod\limits_{1\leq j\neq l\leq k}\theta(t_j-t_l+y)},\\
 w_I^+(t;z,y,\lambda) = \frac{\omega^+_I(t;z,y,\lambda)}{\psi_I(y,\lambda) \prod\limits_{1\leq j\neq l\leq k}\theta(t_j-t_l+y)}.
 \end{gather*}
 \end{Definition}
 \begin{Remark} The factor $1/\psi_I$, defined in Corollary \ref{cor-2}, simplifies the orthogonality relations and the action of the permutations of the $z_i$ at the cost of introducing poles at $\lambda+y\mathbb Z$ modulo
 $\mathbb Z+\tau\mathbb Z$.
 \end{Remark}
 Let $I=\{i_1,\dots,i_k\}\subset[n]$ and $f(t_1,\dots,t_k)$ a symmetric function of $k$ variables. We write $f(z_I)$ for $f(z_{i_1},\dots,z_{i_k})$.
 \begin{Lemma}\label{lemma-4} For each $I,J\subset[n]$ such that $|I|=|J|$, the weight functions $w_I^-(z_J;z,y,\lambda)$ and $\psi_I(y,\lambda) w_I^+(z_J;z,y,\lambda)$ are entire functions of $z$, $y$, $\lambda$.
 \end{Lemma}
 \begin{proof} The vanishing condition implies that $\omega_I^\pm(z_a,z_b,t_3,\dots)$ is divisible by $\theta(z_b-z_a+y)$ so that the quotient by $\theta(t_2-t_1+y)$ is regular at $z_b=z_a-y$ after substitution $t_1=z_a$, $t_2=z_b$. Since $\omega_I^\pm$ is a symmetric function, the same holds for any other pair $t_j$, $t_l$.
 \end{proof}
 The orthogonality relations become:
 \begin{Proposition}[cf.~\cite{RTV,RTV2,RV2}]\label{p-ortho} Let $I,J\subset[n]$, $|I|=|J|=k$. The normalized weight functions obey the orthogonality relations
 \begin{gather*}
 \sum_{K}\frac{w^-_I(z_K,z,y,\lambda)w^+_J(z_K,z,y,\lambda)} {\prod\limits_{a\in K}\prod\limits_{b\in\bar K}\theta(z_a-z_b)\theta(z_a-z_b+y)} =\delta_{I,J}.
 \end{gather*}
 The summation is over subsets $K\subset[n]$ of cardinality $|K|=k$.
\end{Proposition}
\begin{proof} This is a rewriting of Corollary \ref{cor-2} by using Lemma~\ref{lemma-2}.
\end{proof}

We will also use the orthogonality relations in the following equivalent form.
\begin{Corollary} \label{p-ortho2} Let $I,K\subset [n]$, $|I|=|K|=k$. We have
\begin{gather*}
\sum_J w^-_J(z_I,z,y,\lambda) w_J^+(z_K,z,y,\lambda)=
\begin{cases} \prod\limits_{a\in I,\,b\in\bar I}\theta(z_a-z_b)\theta(z_a-z_b+y),& $I=K$,\\
0,& \text{otherwise}.
\end{cases}
\end{gather*}
\end{Corollary}
\begin{proof}Let
\begin{gather*}
x_{IK} =\frac{w^-_I(z_K,z,y,\lambda)}{\prod\limits_{a\in K, \, b\in \bar{K}} \theta(z_a-z_b)}, \qquad
y_{KJ} =\frac{w^+_J(z_K,z,y,\lambda)}{\prod\limits_{a\in K, \, b\in \bar{K}} \theta(z_a-z_b+y)}.
\end{gather*}
Proposition \ref{p-ortho} claims that the matrix $(x_{IK})_{I,K}$ is the left inverse of the matrix $(y_{KJ})_{K,J}$. This implies, however, that the matrix $(x_{IK})_{I,K}$ is also a right inverse of $(y_{KJ})_{K,J}$, which is equivalent to the statement of the corollary.
 \end{proof}

Weight functions have a triangularity property. Introduce a partial ordering on the subsets of $[n]$ of fixed cardinality $k$: if $I=\{i_1<\dots<i_k\}$ and $J=\{j_1<\dots<j_k\}$, then $I\leq J$ if and only if $i_1\leq j_1$, \dots, $i_k\leq j_k$.
\begin{Lemma}\label{lemma-5} Let $\epsilon\colon[n]^2\to\{0,1\}$ be such that
 \begin{gather*}
 \epsilon(a,b)=\begin{cases}1,&\text{if $a>b$},\\
 0,&\text{if $a<b$.}
 \end{cases}
 \end{gather*}
 Then
 \begin{enumerate}\itemsep=0pt
 \item[$(i)$] $w_I^-(z_J;z,y,\lambda)$ vanishes unless $J\leq I$ and
 \begin{gather*}
 w_I^-(z_I;z,y,\lambda)=\prod_{a\in I}\theta(\lambda-(w(a,I)+1)y)\prod_{a\in I,\, b\in \bar I}\theta(z_a-z_b+\epsilon(b,a)y).
 \end{gather*}
 \item[$(ii)$] $w_I^+(z_J,z,y,\lambda)$ vanishes unless $I\leq J$ and
 \begin{gather*}
 w_I^+(z_I;z,y,\lambda)=\frac {\prod\limits_{a\in I,\,b\in \bar I}\theta(z_a-z_b+\epsilon(a,b)y)} {\prod\limits_{a\in I}\theta(\lambda-(w(a,I)+1)y)}.
 \end{gather*}
 \end{enumerate}
\end{Lemma}
\subsection{Eigenvectors of the Gelfand--Zetlin algebra}\label{ss-3.10}
The normalized weight functions $w^-_I$ evaluated at $z_J$ provide the (triangular) transition matrix between the standard basis of $\big(\mathbb C^2\big)^{\otimes n}$ and a basis of eigenvectors of the Gelfand--Zetlin
algebra. The Gelfand--Zetlin algebra is generated by the determinant $\Delta(w)$, see~\eqref{e-determinant}, and $L_{22}(w)$. The determinant acts by multiplication by
\begin{gather*}
 \prod_{i=1}^n\frac{\theta(w-z_i+y)}{\theta(w-z_i)}.
\end{gather*}
 We thus need to diagonalize $L_{22}(w)$.
\begin{Lemma}\label{lemma-6} Let $0\leq k\leq n$, $[k]=\{1,\dots,k\}$. Then
 \begin{gather*}
 \xi_{[k]}= \prod_{i=1}^k\theta(\lambda+(n-k-i)y)v_{[k]}\in V(z_1)\otimes\cdots\otimes V(z_n)
 \end{gather*}
 is an eigenvector of $L_{22}(w)$ with eigenvalue
 \begin{gather*}
 \prod_{a=1}^k\frac{\theta(w-z_a)}{\theta(w-z_a-y)}.
 \end{gather*}
\end{Lemma}
\begin{proof} (See \cite{GRTV,RTV2,RV2}.) Since $L_{21}(w)v_1=0=L_{12}(w)v_2$, the action of $L_{22}(w)$ on $v_1^{\otimes k}\otimes v_2^{\otimes n-k}$ is simply the product of the action on all factors, with the appropriate shift of $\lambda$. Since $L_{22}(w)$ acts diagonally in the basis $v_1$, $v_2$ one gets the result by straightforward calculation.
\end{proof}

For $I\subset[n]$, $|I|=k$, define
\begin{gather}\label{e-xi}
 \xi_I=\xi_I(z,y,\lambda)= \sum_{|J|=k}\frac{w^-_J(z_I,z,y,\lambda)}{\prod\limits_{a\in I,\, b\in\bar I}\theta(z_a-z_b+y)} v_J.
\end{gather}
By Lemma \ref{lemma-5} this definition is consistent with the one for
$\xi_{[k]}$ above.
\begin{Proposition}[{cf.~\cite{GRTV,RTV2,RV2}}]\label{prop-5} The vectors $\xi_I$, $I\subset[n]$, $|I|=k$ form a basis of eigenvectors of the operators of the Gelfand--Zetlin algebra on $V(z_1)\otimes\cdots\otimes V(z_n)$:
 \begin{gather*}
 \Delta(w)\xi_I=\prod_{a=1}^n\frac{\theta(w-z_a+y)}{\theta(w-z_a)} \xi_I, \qquad L_{22}(w)\xi_I= \prod_{a\in I}\frac{\theta(w-z_a)}{\theta(w-z_a-y)} \xi_I.
 \end{gather*}
\end{Proposition}
\begin{proof}
 By Corollary \ref{cor-1}(i), we have that
 \begin{gather*}
 \xi_I(z,y,\lambda) =S_i(z,y,\lambda)\xi_{s_i\cdot I}(s_iz,y,\lambda).
 \end{gather*}
Thus $\xi_I(z,y,\lambda)$ is related to $\xi_{s_i\cdot I}(s_iz,y,\lambda)$ by a morphism of representations of the elliptic dynamical quantum group. If $|I|=k$ then there is a permutation $\sigma$ such that $\sigma\cdot I=[k]$ and thus $\xi_I(z,y,\lambda)=\rho(\sigma)\,\xi_{[k]}(\sigma\cdot z,y,\lambda)$ for some morphism $\rho(\sigma)$. Since $\xi_{[k]}(z,y,\lambda)$ is an eigenvector of~$L_{22}(w)$ with eigenvalue~$\mu_{[k]}(w;z,y)$, see Lemma~\ref{lemma-6}, we deduce that $\xi_I(z,y,\lambda)$ is an eigenvector with eigenvalue $\mu_I(w;z,y)=\mu_{[k]}(w;\sigma\cdot z,y)$.
\end{proof}
\subsection{An explicit formula for weight functions}\label{ss-expli}
Let $I=\{i_1<i_2<\cdots<i_k\}\subset [n]$ and recall the definition \eqref{eq-w} of $w(i,I)$ and of $\psi_I$ in Corollary~\ref{cor-2}. We set
\begin{gather*}
 l^+_I(r,a;t,z,y,\lambda) =
 \begin{cases}
 \theta(t_r-z_a+y),& \text{if $a<i_r$}, \\
 \theta(\lambda+t_r-z_a-w(i_r,I)y), & \text{if $a=i_r$}, \\
 \theta(t_r-z_a), & \text{if $a>i_r$},
 \end{cases}
\end{gather*}
and
\begin{gather*}
 l^-_I(r,a;t,z,y,\lambda) =
 \begin{cases}
 \theta(t_r-z_a),& \text{if $a<i_r$}, \\
 \theta(\lambda-t_r+z_a-(w(i_r,I)+1)y), & \text{if $a=i_r$}, \\
 \theta(t_r-z_a+y), & \text{if $a>i_r$}.
 \end{cases}
\end{gather*}
Then
\begin{gather*}
 w^+_I(t;z,y,\lambda)=\frac{1}{\psi_I(y,\lambda)}\operatorname{Sym} \left( \frac { \prod\limits_{r=1}^k\prod\limits_{a=1}^nl^+_I(r,a;t,z,y,\lambda) } {\prod\limits_{1\leq i<j\leq k}\theta(t_i-t_j)\theta(t_j-t_i+y)}\right),
\end{gather*}
and
\begin{gather*}
 w^-_I(t;z,y,\lambda)=\operatorname{Sym}\left( \frac { \prod\limits_{r=1}^k\prod\limits_{a=1}^nl^-_I(r,a;t,z,y,\lambda) } {\prod\limits_{1\leq i<j\leq k}\theta(t_j-t_i)\theta(t_i-t_j+y)}\right).
\end{gather*}
\subsection{Dual bases and resonances}\label{ss-3.12}
Here we prove Propositions \ref{prop-1} and \ref{prop-2}. They are corollaries of the following more precise statement:
\begin{Proposition}\label{prop-6} Let $\Lambda=\mathbb Z+\tau\mathbb Z$ and fix $(z,y,\lambda)\in\mathbb C^n\times \mathbb C\times \mathbb C$. Assume that $z_a-z_b-sy\not\in\Lambda$, for all $a\neq b$, $s=0,\dots,k$, $\lambda-sy\not\in\Lambda$ for all $s\in S\subset\mathbb Z$ for some finite $S$ depending on $k$ and~$n$.
 \begin{enumerate}\itemsep=0pt
 \item[$(i)$] Let $\omega^\pm_{k_a}=\omega_{k_a}^\pm\Big(t;z_a,\lambda+y\sum\limits_{b=a+1}^n(1-2k_b)\Big)$, see Example~{\rm \ref{example-4}}. Then the family
 \begin{gather*}
 \omega^\pm_{k_1}*\dots*\omega^\pm_{k_n}, \qquad k\in\mathbb Z^k_{\geq0},\qquad\sum_{a=1}^n k_a=k
 \end{gather*}
 is a basis of $\Theta^\pm_k(z,y,\lambda)$.
 \item[$(ii)$] The subfamily indexed by $(k_a)_{a=1}^n$ such that $k_a\in\{0,1\}$ for all $a$ is a basis of $\bar\Theta^\pm_k(z,y,\lambda)$.
 \end{enumerate}
\end{Proposition}
Part (i) and a special case of (ii) are proved in~\cite{FelderTarasovVarchenko1997}. The proof relies on the following construction of linear forms whose evaluations on the members of the family form a non-degenerate triangular matrix. For a symmetric function $f(t_1,\dots,t_k)$ and $w\in\mathbb C$, define $\mathrm{ev}_wf$ to be the symmetric function of $k-1$ variables
\begin{gather*}
 \mathrm{ev}_wf(t_1,\dots,t_{k-1})=f(t_1,\dots,t_{k-1},w).
\end{gather*}
It is easy to check that $\mathrm{ev}_w$ maps $\Theta^\pm_k(z,y,\lambda)$ to $\Theta^\pm_{k-1}(z,y,\lambda\pm y)$. For $\ell\in\mathbb Z_{\geq0}$ and $f$ a symmetric function of $k$ variables set
\begin{gather*}
 \mathrm{ev}_{w,\ell}(f)=\begin{cases}
 f&\text{if $\ell=0$},\\
 \mathrm{ev}_{w-\ell y}\circ\cdots\circ\mathrm{ev}_{w-y}\circ\mathrm{ev}_w(f) &\text{if $1\leq\ell\leq k$},\\
 0,&\text{otherwise}.
 \end{cases}
\end{gather*}
Finally, for $\ell\in\mathbb Z_{\geq0}^n$, $\sum\limits_{a=1}^n\ell_a=k$ we introduce linear forms $\epsilon_{\ell_1,\dots,\ell_n}\in\Theta^\pm(z,y,\lambda)^*$:
\begin{gather*}
 \epsilon_{\ell_1,\dots,\ell_n}= \mathrm{ev}_{z_n,\ell_n}\circ\cdots\circ\mathrm{ev}_{z_1,\ell_1}.
\end{gather*}
The following result is a special case of Proposition~30 from~\cite{FelderTarasovVarchenko1997} (adapted to the conventions of this paper). It can be checked by induction using the fact that the evaluation points are such that at most one shuffle in the definition of the shuffle products contributes nontrivially.
\begin{Lemma}\label{lemma-7} Let $(z,y,\lambda)\in\mathbb C^n\times \mathbb C\times\mathbb C$, $k,\ell\in\mathbb Z_{\geq0}^n$ with $\sum\limits_{a=1}^n\ell_a= \sum\limits_{a=1}^n k_a$. Then
 \begin{enumerate}\itemsep=0pt
 \item[$(i)$] Let $f_a\in\Theta^-_{k_a}\Big(z_a,y,\lambda-y\sum\limits_{b=a+1}^n(2k_b-1)\Big)$, $a=1,\dots,n$. Then
 \begin{gather*}
 \epsilon_{\ell_1,\dots,\ell_n}(f_1*\cdots*f_n)=0,
 \end{gather*}
 unless $\ell_1+\cdots+\ell_p\leq k_1+\cdots+k_p$ for all $p=1,\dots,n$, and
 \begin{gather*}
 \epsilon_{k_1,\dots,k_n} (f_1*\cdots*f_n) \\
\qquad {} =\prod_{a=1}^n\mathrm{ev}_{z_a,k_a}(f_a) \prod_{a<b}\left(\prod_{s=0}^{k_b-1}\theta(z_b-z_a-ys) \prod_{s=0}^{k_a-1}\theta(z_a-z_b+y(1-s))\right).
 \end{gather*}
 \item[$(ii)$] Let $\omega^-_k$ be the basis of $\Theta^-_k(z,y,\lambda)$, $z\in\mathbb C$, defined in Example~{\rm \ref{example-4}}. Then
 \begin{gather*}
 \mathrm{ev}_{z,k}\omega^-_k=\prod_{s=1}^k\theta(\lambda-sy).
 \end{gather*}
 \end{enumerate}
\end{Lemma}

Setting $f_a=\omega^-_{k_a}$, $a=1,\dots, n$ gives a proof of Proposition~\ref{prop-6} (i) in the case of $\Theta^-$. The case of $\Theta^+$ is reduced to this case by Proposition~\ref{p-duality}.

We turn to the proof of Proposition~\ref{prop-6}(ii). In the notation we have introduced here, $\bar\Theta^\pm_k(z,y,\lambda)$ is the intersection of the kernels of $\mathrm{ev}_{z_a,2}$ for $a=1,\dots,n$.

Let $(z,y,\lambda)\in\mathbb C^n\times\mathbb C\times \mathbb C$ and $1\leq c\leq n$. In the following proposition we describe the interaction of $\mathrm{ev}_{z_c}$ and $\mathrm{ev}_{z_c-y}$ with the shuffle product. By using the identifications of Remark~\ref{remark-1}, we view these maps as maps between the following spaces:
\begin{gather*}
 \mathrm{ev}_{z_c} \colon \ \Theta^\pm_k(z,y,\lambda)\to \Theta^\pm_{k-1}(z_1,\dots,z_c-y,\dots,z_n,y,\lambda),\\
 \mathrm{ev}_{z_c-y} \colon \ \Theta^\pm_k(z,y,\lambda)\to \Theta^\pm_{k-1}(z_1,\dots,z_c+y,\dots,z_n,y,\lambda\pm 2y).
\end{gather*}
\begin{Proposition}\label{prop-ev-*} In the notation of Proposition~{\rm \ref{prop-0}}, let $f\in\Theta^-_{k'}(z',y,\lambda+y(n''-2k''))$, $g\in \Theta^-_{k''}(z'',y,\lambda)$. We have
\begin{enumerate}\itemsep=0pt
\item[$(i)$] $ \mathrm{ev}_{z_c}(f*g)= \mathrm{ev}_{z_c}(f)*g \prod\limits_{b=n'+1}^n\theta(z_c-z_b+y)$, $1\leq c\leq n'$,
\item[$(ii)$] $\mathrm{ev}_{z_c-y}(f*g)= f*\mathrm{ev}_{z_c-y}(g) \prod\limits_{a=1}^{n'}\theta(z_c-z_a-y)$, $n'<c\leq n$.
\end{enumerate}
Similarly, let $f\in\Theta^+_{k'}(z',y,\lambda+y(n''-2k''))$, $g\in \Theta^+_{k''}(z'',y,\lambda)$. We have
\begin{enumerate}\itemsep=0pt
\item[$(iii)$] $ \mathrm{ev}_{z_c-y}(f*g)= \mathrm{ev}_{z_c-y}(f)*g \prod\limits_{b=n'+1}^n\theta(z_c-z_b+y)$, $1\leq c\leq n'$,
\item[$(iv)$] $ \mathrm{ev}_{z_c}(f*g)= f*\mathrm{ev}_{z_c}(g) \prod\limits_{a=1}^{n'}\theta(z_c-z_a-y)$, $n'<c\leq n$.
\end{enumerate}
\end{Proposition}
\begin{proof} (i) Due to the factor $\prod\limits_{l=k'+1}^k\theta(t_l-z_c)$ in the definition of $\varphi^-$, see Proposition~\ref{prop-0}, the only terms in the sum over permutations contributing nontrivially to $\mathrm{ev}_{z_c}(f*g)$ are such that~$t_k$ is an argument of~$f$. Thus $\mathrm{ev}_{z_c}(f*g)=\mathrm{ev}_{z_c}(f)*g$ times a~factor that is computed explicitly. The proof of (ii)--(iv) is similar.
\end{proof}

By iterating, we obtain:
\begin{Corollary}\label{cor-3} Let $f\in\Theta^-_{k'}(z',y,\lambda+y(n''-2k''))$, $g\in \Theta^-_{k''}(z'',y,\lambda)$. We have
 \begin{gather*}
 \mathrm{ev}_{z_c,2}(f*g)= \mathrm{ev}_{z_c,2}(f)*g \prod_{b=n'+1}^n\theta(z_c-z_b+y)\theta(z_c-z_b), \qquad 1\leq c\leq n', \\
 \mathrm{ev}_{z_c,2}(f*g) = f*\mathrm{ev}_{z_c,2}(g) \prod_{a=1}^{n'}\theta(z_c-z_a-y)\theta(z_c-z_a), \qquad n'<c\leq n,
 \end{gather*}
 and similarly for $\Theta^+$.
 In particular, if $f$ and $g$ satisfy the vanishing condition then
 also $f*g$ does.
\end{Corollary}

\begin{proof}[Proof of Proposition \ref{prop-6}(ii)] We give the proof for $\bar\Theta^-_k(z,y,\lambda)$. The proof for $\bar\Theta^+_k(z,y,\lambda)$ is similar or can be deduced using the duality map of Proposition~\ref{p-duality}. It follows from Corollary~\ref{cor-3} that the indicated subfamily does belong to $\bar\Theta^-_k(z,y,\lambda)$. It remains to show that it is a~spanning set. By Proposition~\ref{prop-6}, (i), we know that any element of $\bar\Theta^-_k(z,y,\lambda)$ can be written as linear combinations
\begin{gather*}
 \sum_{k_1+\cdots+k_n=k}\lambda_{k_1,\dots,k_n} \omega_{k_1}*\cdots*\omega_{k_n}.
\end{gather*}
On the other hand, the linear form $\epsilon_{k_1,\dots,k_n}$ vanishes on $\bar \Theta^-_k(z,y,\lambda)$ if $k_a\geq2$ for some $a$, since it involves the evaluation at $z_a,z_a-y$. By Lemma~\ref{lemma-8}, the coefficients $\lambda_{k_1,\dots,k_n}$ must thus vanish if at least one $k_a$ is $\geq2$ which is what we had to prove.
\end{proof}

\section{Equivariant elliptic cohomology of Grassmannians}\label{s-4}

Let $E$ be an elliptic curve and $G$ a compact group. Equivariant elliptic cohomology was postulated by Ginzburg, Kapranov and Vasserot in \cite{GKV} as a functor $E_G$ from pairs of finite $G$-CW complexes to superschemes satisfying a set of axioms, generalizing those satisfied by equivariant cohomology and equivariant $K$-theory. One of them being that for a point pt, $E_G(\mathrm{pt})$ is a suitable moduli scheme of $G$-bundles on the dual elliptic curve. For example $E_{U(n)}(\mathrm{pt})=E^{(n)}=E^n/S_n$ and for an abelian group $A$ with group of characters $X(A)=\operatorname{Hom}(A,U(1))$, $E_A(\mathrm{pt})=\operatorname{Hom}(X(A),E)$. By functoriality, the scheme $E_G(M)$ comes with a structure map
\begin{gather*}
 p_G\colon \ E_G(M)\to E_G(\mathrm{pt}).
\end{gather*}
For a complex elliptic curve, the case we consider here, a~construction of equivariant elliptic cohomology was given by Grojnowski \cite{GKV,Grojnowski}. It has the property that for a~connected Lie group~$G$ with maximal torus $A$ and Weyl group $W$ then $W$ acts on $E_A(M)$ and $E_G(M)=E_A(M)/W$.

\subsection{Tautological bundles and Chern classes}\label{ss-eecG}
Let $E$ be a complex elliptic curve. The unitary group $U(n)$ and its maximal torus $A\cong U(1)^n$ act on the Grassmannian $\mathrm{Gr}(k,n)$ of $k$-dimensional subspaces of~$\mathbb C^n$. The $A$-equivariant cohomo\-logy of $\mathrm{Gr}(k,n)$ was described in \cite[Section~1.9]{GKV}, and is analogous to the classical description of ordinary equivariant cohomology in terms of Chern classes of tautological bundles. The Grassmannian has two tautological equivariant vector bundles of rank $k$ and $n-k$, respectively. They give rise to a characteristic map~\cite{GKV}
\begin{gather*} \chi\colon \ E_A(\mathrm{Gr}(k,n))\to E_{U(k)}(\mathrm{pt})\times E_{U(n-k)}(\mathrm{pt})=E^{(k)}\times E^{(n-k)}.
\end{gather*} Here $E^{(k)}=E^k/S_k$ denotes the symmetric power of the elliptic curve, which is the $U(k)$-equivariant cohomology of a point. Together with the structure map to $E_A(\mathrm{pt})$ we have a~description of
the equivariant elliptic cohomology as the fiber product of $E^{(k)}\times E^{(n-k)}$ and $E^n$ over $E^{(n)}$, namely, we have the Cartesian square:
\begin{gather*}
 \begin{array}{@{}ccc}
 E_A(\mathrm{Gr}(k,n))&\longrightarrow&E^{(k)}\times E^{(n-k)}\\
 \downarrow& &\downarrow\\ E^n&\longrightarrow&E^{(n)}.
 \end{array}
\end{gather*}
The left vertical arrow is the structure map to $E_A(\mathrm{pt})$; the maps $E^n\to E^{(n)}$ and $E^{(k)}\times E^{(n-k)}\to E^{(n)} $ are the canonical projections. Thus $ E_A(\mathrm{Gr}(k,n))=(E^{(k)}\times
E^{(n-k)})\times_{E^{(n)}}E^n $.

The symmetric group $S_n$ (the Weyl group of $U(n)$) acts on the diagram above (with trivial action on the right column) and the $U(n)$-equivariant cohomology is the quotient by this action:
\begin{gather*}
 E_{U(n)}(\mathrm{Gr}(k,n))=E_A(\mathrm{Gr}(k,n))/S_n=E^{(k)}\times E^{(n-k)}.
\end{gather*}

\subsection{Moment graph description} An alternative useful description of the equivariant elliptic cohomology is via the localization theorem, proved by Goresky, Kottwitz and MacPherson \cite{GKM} for equivariant cohomology and generalized to elliptic cohomology by Knutson and Ro\c{s}u~\cite{Rosu}. For partial flag varieties such as Grassmannians it is described explicitly in \cite[Example~4.4]{Ganter}. The action of $A$ on the Grassmannian
$\mathrm{Gr}(k,n)$ has isolated fixed points labeled by subsets of $[n]=\{1,\dots,n\}$ with $k$ elements. The fixed point $x_I$ labeled by $I\subset[n]$ is the $k$-plane spanned by the coordinate axes indexed by~$I$. The inclusion of the fixed points $x_I$ induces a map $\iota_I\colon E_A(\mathrm{pt})\to E_A(\mathrm{Gr}(k,n))$ and it turns out that $E_A(\mathrm{Gr}(k,n))$ is the union of the $\iota_{I}E_A(\mathrm{pt})\simeq E^n$ where~$I$ runs over the subsets of~$[n]$ with~$k$ elements.

Let $\Gamma$ be the graph with vertex set $\Gamma_0$ the set of subsets $I\subset [n]$ with $|I|=k$ elements and an edge connecting $I$ with $I'$ for each pair of vertices such that $|I\cap I'|=k-1$. In this case $I=J\cup\{a\}$ and $I'=J\cup\{b\}$ with $|J|=k-1$ and we set $\Delta_{I,I'}=\{z\in E^n, z_a=z_b\}$. We then have inclusion maps $\Delta_{I,I'}\to \iota_{I}E_A(\mathrm {pt})=E^n$, $\Delta_{I,I'}\to \iota_{I'}E_A(\mathrm {pt})=E^n$.
\begin{Proposition}\label{p-001} We have the coequalizer diagram
 \begin{gather*}
 \sqcup_{|I\cap I'|=k-1}\Delta_{I,I'}\rightrightarrows \sqcup_{I\subset [n],|I|=k}E^n\to E_{A}(\mathrm{Gr}(k,n))).
 \end{gather*}
\end{Proposition}
In other words, $E_A(\mathrm{Gr}(k,n))$ is the union of copies of $E^n$ labeled by subsets $I\subset [n]$ of size $k$, glued along the diagonals $\Delta_{I,I'}$. The structure map $E_A(\mathrm{Gr}(k,n))\to E^n$ is the identity on each copy. The isomorphism between the two descriptions of $E_A(\mathrm{Gr}(k,n))$ is induced by the map
\begin{gather*}
 \sqcup_{I\subset [n],|I|=k}E^n\to \big(E^{(k)}\times E^{(n-k)}\big)\times_{E^{(n)}}E^n,
\end{gather*}
whose restriction to the copy $E^n$ labeled by $I$ is
\begin{gather*} z\mapsto (z,z_I,z_{\bar I}),\qquad z_I=(z_i)_{i\in I}, \qquad z_{\bar I}=(z_j)_{j\in \bar I}.
\end{gather*}
It is easy to check directly Proposition~\ref{p-001} using the fiber product as a definition of the equivariant elliptic cohomology.
\subsection{Cotangent bundles and dynamical parameter}\label{ss-Extended}
The action of $U(n)$ on the Grassmannian induces an action on its cotangent bundle $X_{k,n}=T^*\mathrm{Gr}(k,n)$. Additionally we have an action of $U(1)$ on the cotangent bundle by scalar multiplication on the fibers, so we get an action of
\begin{gather*}
 G=U(n)\times U(1)
\end{gather*}
and its Cartan torus
\begin{gather*}
 T=A\times U(1)\cong U(1)^{n+1}.
\end{gather*}
Since the cotangent bundle is equivariantly homotopy equivalent to its zero section, the equivariant elliptic cohomology is simply
\begin{gather*}
 E_T(X_{k,n})=E_A(\mathrm{Gr}(k,n))\times E,
\end{gather*}
a scheme over $E_T(\mathrm{pt})=E^{n}\times E$, and
\begin{gather*}
 E_G(X_{k,n})=E_{U(n)}(\mathrm{Gr}(k,n))\times E,
\end{gather*}
a scheme over $E_G(\mathrm{pt})=E^{(n)}\times E$.

We will consider, as in \cite{AO}, an extended version of elliptic cohomology to accommodate for dynamical variables in quantum group theory, namely
\begin{gather*}
 \hat E_T(X_{k,n}):=E_T(X_{k,n})\times (\mathrm{Pic}(X_{k,n})\otimes_{\mathbb Z}E)\cong E_T(X_{k,n})\times E,
\end{gather*}
a scheme over $\hat E_T(\mathrm{pt})=E^n\times E\times E$ (the Picard group of the Grassmannian is infinite cyclic generated by the top exterior power of the tautological bundle). Similarly, we set
\begin{gather*}
 \hat E_G(X_{k,n})=E_G(X_{k,n})\times E,
\end{gather*}
which is a scheme over $\hat E_G(\mathrm{pt})=E^{(n)}\times E\times E$.

The fixed points $x_K$ for the $A$-action on the Grassmannian are also isolated fixed points in the cotangent bundle of the Grassmannian for the $T$-action and we have maps $\iota_K=\hat E_T(i_K)\colon \hat E_T(\mathrm{pt})$ $\to \hat E_T(X_{k,n})$ induced by the inclusion $i_K\colon \mathrm{pt} \mapsto x_K$. Then $\hat E_T(X_{k,n})$ consists of the components $\iota_K\hat E_T(\mathrm{pt})$, where $K$ runs over the subsets of $[n]$ with $k$ elements. By Section~\ref{ss-eecG}, we have a~description of $\hat E_T(X_{k,n})$ as a~fiber product:
\begin{gather*}
\hat E_T(X_{k,n})\cong \big(E^{(k)}\times E^{(n-k)}\big)\times_{E^{(n)}}E^n\times E\times E.
\end{gather*}
In particular we have the {\em characteristic embedding}
\begin{gather}\label{e-ce}
c\colon \ \hat E_T(X_{k,n})\to E^{(k)}\times E^{(n-k)}\times E^{n}\times E\times E
\end{gather}
of the extended $T$-equivariant elliptic cohomology scheme into a non-singular projective variety.

\section[Admissible line bundles on $\hat E_T(X_{k,n})$]{Admissible line bundles on $\boldsymbol{\hat E_T(X_{k,n})}$}\label{s-5}
\subsection[Line bundles on $E^{p}$]{Line bundles on $\boldsymbol{E^{p}}$}\label{ss-alb1}
Line bundles on complex tori are classified by the Appel--Humbert theorem, see \cite[Section~I.2]{Mumford}. In the special case of powers of generic elliptic curves this reduces to the following explicit description: let $E=\mathbb C/\Lambda$ with $\Lambda=\mathbb Z+\tau\mathbb Z$ so that $E^p=\mathbb C^p/\Lambda^p$. For each pair $(N,v)$ consisting of a symmetric integral $p\times p$ matrix $N$ and $v\in(\mathbb C/\mathbb Z)^p$ let $ \mathcal L(N,v)$ be the line bundle \mbox{$(\mathbb C^p\times \mathbb C)/\Lambda^p\to E^p$} with action
\begin{gather*}
 \lambda\cdot(x,u)=(x+\lambda,e_\lambda(x)u),\qquad \lambda\in\Lambda^p, \qquad x\in \mathbb C^p,\qquad u\in \mathbb C,
\end{gather*}
and cocycle
\begin{gather*}
 e_{n+m\tau}(x)=(-1)^{n^tNn}(-{\rm e}^{{\rm i}\pi\tau})^{m^tNm}{\rm e}^{2\pi{\rm i} m^t(Nx+v)}, \qquad n,m\in\mathbb Z^p.
\end{gather*}
\begin{Proposition} \label{p-linebundles} \quad
 \begin{itemize}\itemsep=0pt
 \item[$(i)$] $\mathcal L(N,v)$ is isomorphic to $\mathcal L(N',v')$ if and only if $N=N'$ and $v\equiv v'\mod\Lambda^p$.
 \item[$(ii)$] For generic $E$, every holomorphic line bundle on $E^p$ is isomorphic to $\mathcal L(N,v)$ for some $(N,v)$.
 \item[$(iii)$] $\mathcal L(N_1,v_1)\otimes \mathcal L(N_2,v_2)\cong \mathcal L(N_1+N_2,v_1+v_2)$.
 \item[$(iv)$] Let $\sigma\in S_p$ act by permutations on $E^p$ and $\mathbb C^p$. Denote also by $\sigma$ the corresponding $p\times p$ permutation matrix. Then
 \begin{gather*}\sigma^*\mathcal L(N,v)=\mathcal L\big(\sigma^t N\sigma,\sigma^t v\big).\end{gather*}
 \end{itemize}
\end{Proposition}

To an integral symmetric $p\times p$ matrix $N$ and a vector $v\in \mathbb C^p$ we associate the integral quadratic form $N(x)=x^tNx$ and the linear form $v(x)=x^tv$ on the universal covering $\mathbb C^p$ of $E^p$ and we call them the quadratic form and the linear form of the line bundle $\mathcal L(N,v)$. The linear form is defined up to addition of an integral linear form.

\begin{Remark} Exceptions to (ii) are elliptic curves with complex multiplication, in which case there are additional line bundles that are not isomorphic to those of the form $\mathcal L(N,v)$.
\end{Remark}
\begin{Remark} The map $E^p\to \mathrm{Pic}(E^p)$ sending $v$ to $\mathcal L(0,v)$ is an isomorphism onto the subgroup $\mathrm{Pic}^0(E^p)$ of classes of line bundles of degree~0. If $E$ is a generic elliptic curve, the N\'eron--Severi group $\mathrm{NS}(E^p)=\mathrm{Pic}(E^p)/\mathrm{Pic}^0(E^p)$ is a free abelian group of rank $n(n+1)/2$ identified with the group of integral symmetric matrices via $N\mapsto \mathcal L(N,0)$.
\end{Remark}
\begin{Remark} \label{rem5.4} Sections of $\mathcal L(N,v)$ are the same as functions $f$ on $\mathbb C^p$ such that $f(x+\lambda)=e_\lambda(x)^{-1}f(x)$ for all $\lambda\in\Lambda^p$, $x\in\mathbb C^p$. Explicitly, a function on $\mathbb C^p$ defines a section of $\mathcal L(N,v)$ if and only if
 \begin{gather*} f(x_1,\dots,x_j+1,\dots,x_p) = (-1)^{N_{jj}}f(x), \\
 f(x_1,\dots,x_j+\tau,\dots,x_p) = (-1)^{N_{jj}}{\rm e}^{-2\pi{\rm i}(\sum_k N_{jk}x_k+v_j)-\pi {\rm i} \tau N_{jj}} f(x),
 \end{gather*} for all $x\in\mathbb C^p$, $j=1,\dots,p$.
\end{Remark}
\begin{Remark} The factors of $-1$ in the cocycle can be removed by going to an equivalent cocycle. With the present convention the line bundles $\mathcal L(N,0)$ correspond to divisors whose irreducible components are subgroups. Let $\theta(z)$ be the odd Jacobi theta function in one variable, see~\eqref{e-theta}. Then, for any $r\in\mathbb Z^p$ and $z\in\mathbb C$,
 \begin{gather*}
 \theta(r^tx+z)=\theta(r_1x_1+\cdots+r_px_p+z)
 \end{gather*}
 is a holomorphic section of $\mathcal L(N,v)$ with quadratic form
 \begin{gather*}
 N(x)=\left(\sum_{i=1}^pr_ix_i\right)^2,
 \end{gather*}
 and linear form
 \begin{gather*}
 v(x)=z\sum_{i=1}^p r_ix_i.
 \end{gather*}
If $z=0$ this section vanishes precisely on the subtorus $\mathrm{Ker}(\phi_r)$, the kernel of the group homomorphism $\phi_r\colon E^p\to E$, $x\mapsto \sum_ir_ix_i$. Since an integral quadratic form is an integral linear combination of squares of integral linear forms, $\mathcal L(N,0)$ has a meromorphic section which is a~ratio of products of theta functions $\theta(r^tx)$ with $r\in\mathbb Z^p$.
\end{Remark}
\subsection{Admissible line bundles}
The elliptic dynamical quantum group acts on sections of admissible line bundles, which are, up to a~twist by a fixed line bundle, those coming from the base scheme $\hat E_{T}(\mathrm{pt})$. Let $p_T$ be the structure map
\begin{gather*}
 p_T\colon \ \hat E_T(X_{k,n})\to \hat E_T(\mathrm{pt}),
\end{gather*}
and $\hat\chi=\chi\times \mathrm{id}\times\mathrm{id}$ the characteristic map
\begin{gather*}
 \hat\chi \colon \ \hat E_T(X_{k,n})=E_A(\mathrm{Gr}(k,n))\times E\times E\to E^{(k)}\times E^{(n-k)}\times E\times E.
\end{gather*}
Let $t_1,\dots,t_k,s_1,\dots,s_{n-k},y,\lambda$ be coordinates on the universal covering of $E^k\times E^{n-k}\times E\times E$ and $N$ the quadratic form
\begin{gather}\label{e-Nkn}
 N_{k,n}(t,s,y,\lambda)=2\sum_{i=1}^k t_i(\lambda+(n-k)y)+\sum_{i=1}^k\sum_{j=1}^{n-k}(t_i-s_j)^2.
\end{gather}
Clearly $N_{k,n}$ is symmetric under permutation of the coordinates $t_i$ and of the coordinates $s_j$ and thus $\mathcal L(N_{k,n},0)$ can be considered as a bundle on $E^{(k)}\times E^{(n-k)}\times E\times E$.
\begin{Definition}\label{def-tlb}
 The {\em twisting line bundle} on $X_{k,n}$ is $\mathcal T_{k,n}=\hat\chi^*\mathcal L(N_{k,n},0)$
\end{Definition}
\begin{Definition} An {\em admissible line bundle} on $\hat E_T(X_{k,n})$ is a line bundle of the form
 \begin{gather*}
 p_T^*\mathcal L\otimes \mathcal T_{k,n},
 \end{gather*}
 for some line bundle $\mathcal L$ on $\hat E_T(\mathrm{pt})$.
\end{Definition}
\subsection{Holomorphic and meromorphic sections}\label{ss-ms}
We will consider meromorphic sections of line bundles on elliptic cohomology schemes. Since these schemes are singular, we need to be careful about the definition. Recall that $\hat E_T(X_{k,n})$ has components $Y_K=\iota_K\hat E_T(\mathrm{pt})$, corresponding to the inclusion of the fixed points $x_K$, labeled by subsets $K\subset[n]$ of cardinality $k$. We say that a meromorphic section on a complex manifold {\em restricts to a meromorphic section} on a submanifold if it is defined at its generic point, i.e., if the divisor of poles does not contain a component of the submanifold.
\begin{Definition}\label{def-mero} Let $\mathcal L$ be a line bundle on $\hat E_T(X_{k,n})$. A {\em meromorphic section} of $\mathcal L$ is a collection of meromorphic sections $s_I$ of $\mathcal L|_{Y_I}$, labeled by $I\subset[n]$ with $|I|=k$ and restricting to meromorphic sections on all intersections $Y_{I_1}\cap\cdots\cap Y_{I_s}$ and such that
 \begin{gather*}
 s_I|_{Y_I\cap Y_J}=s_J|_{Y_I\cap Y_J},
 \end{gather*}
for all $I$, $J$. A {\em holomorphic section} is a meromorphic section whose restriction to each $Y_I$ is holomorphic. We denote by $\Gamma(\hat E_T(X_{k,n}),\mathcal L)$ the space of holomorphic sections of $\mathcal L$ and by $\Gamma_{\mathrm{mer}}(\hat E_T(X_{k,n}),\mathcal L)$ the space of meromorphic sections of $\mathcal L$.
\end{Definition}

\subsection{Weight functions and admissible line bundles}
With the description of line bundles of Section~\ref{ss-alb1}, the weight functions $w^+_I(t_1,\dots,t_k,z_1,\dots,z_n$, $y,\lambda)$ can be viewed as sections of certain line bundles on $E^{(k)}\times E^{(n-k)}\times E^n\times E\times E$, namely as ($s$-independent) functions of the coordinates $(t_1,\dots,t_k,s_1,\dots,s_{n-k},z,y,\lambda)$ on the universal covering space, with proper multipliers under lattice translations. Their pull-back by the characteristic embedding
\begin{gather*}
 c \colon \ \hat E_T(X_{k,n})\hookrightarrow E^{(k)}\times E^{(n-k)}\times E^n\times E\times E,
\end{gather*}
see \eqref{e-ce}, is a section of the pull-back bundle and its restriction to $\iota_J\hat E_T(\mathrm{pt})$ is the evaluation of~$w^+_I$ at~$t=z_I$.

\begin{Proposition}\label{prop-8} Let $I\subset[n]$, $|I|=k$. Then the restriction $c ^*w^+_I$ of $w^+_I$ to $\hat E_T(X_{k,n})$ is a~meromorphic section of the admissible bundle $p_T^*\mathcal L_I\otimes \mathcal T_{k,n}$ for some line bundle $\mathcal L_I$ on $\hat E_T(\mathrm{pt})$. Moreover $\psi_Ic ^*w^+_I$ is holomorphic.
\end{Proposition}
\begin{proof} We need first to check that all terms in the sum over $S_k$ defining the symmetrization map Sym in Section~\ref{ss-expli} have the same transformation properties under shifts of the variables by the lattice so that they define sections of the same line bundle on $E^k\times E^n\times E^2$. The symmetrization map then produces a section symmetric under permutations of $t_i$, which is the same as a section of a line bundle on
 $E^{(k)}\times E^n\times E^2$. The transformation properties are encoded in the quadratic form: the argument of $\operatorname{Sym}$ is a~section of the line bundle $\mathcal L(M_I,0)$ with
 \begin{gather*}
 M_I(t,z,y,\lambda) =2\sum_{r=1}^k t_r(\lambda+(n-k)y)+\!\sum_{r=1}^k\sum_{a=1}^n (t_r-z_a)^2-2\!\!\!\sum_{1\leq r<s\leq k}\!\!(t_r-t_s)^2 + \tilde M_I(z,y,\lambda),
 \end{gather*}
where $\tilde M_I$ is independent of $t_1,\dots,t_k$. Since $M_I$ is symmetric under permutations of the variables~$t_i$ it defines an $S_n$-equivariant line bundle. All terms in the sum over permutations are sections of this line bundle and their sum is a symmetric section, i.e., the pull-back of a section on the quotient $E^{(k)}\times E^n\times E^2$, which we understand as a section on $E^{(k)}\times E^{(n-k)}\times E^n\times E^2$, constant along $E^{(n-k)}$. The restriction to the component $\iota_K\hat E_T(\mathrm{pt})$ of $\hat E_T(X_{k,n})$ is $w^+_I(z_K,z,y,\lambda)$, the result of substituting the variables $t_i$ by $z_K=(z_i)_{i\in K}$. It is a section of the line bundle with quadratic form
 \begin{gather*}
 M_I(z_K,z,y,\lambda) = 2\sum_{i\in K}z_i(\lambda+(n-k)y) +\sum_{i\in K}\sum_{a=1}^n(z_i-z_a)^2 -\sum_{i,j\in K}(z_i-z_j)^2+\tilde M_I(z,y,\lambda)\\
\hphantom{M_I(z_K,z,y,\lambda)}{} =2\sum_{i\in K}z_i(\lambda+(n-k)y) +\sum_{i\in K,\, j\in\bar K}(z_i-z_j)^2 +\tilde M_I(z,y,\lambda)\\
\hphantom{M_I(z_K,z,y,\lambda)}{}= N_{k,n}(z_K,z_{\bar K},y,\lambda)+\tilde M_I(z,y,\lambda),
 \end{gather*}
 cf.~\eqref{e-Nkn}. Thus the symmetrization is a section of the tensor product of the twisting bundle and the bundle with quadratic form $\tilde M_I$ which is independent of $K$ and thus comes from $\hat E_T(\mathrm{pt})=E^n\times E^2$. The section $\psi_Ic ^*w^+_I$ is holomorphic because of Lemma~\ref{lemma-4}.
\end{proof}

Thus $c ^*w_I^+$ is a meromorphic section of an admissible line bundle $p_T^*\mathcal L_I\otimes \mathcal T_{k,n}$ with poles on a~finite set of hypertori with equation $\lambda-jy=0$, $j\in\mathbb Z$, the divisors of zeros of $\psi_I$. The bundle~$\mathcal L_I$ can be calculated: $\mathcal L_I=\mathcal L(N_I,0)$ with
\begin{gather}
 N_I=-2\sum_{a\in\bar I}n(a,I)z_ay-2\sum_{a\in I}z_a(\lambda+n(a,\bar I)y)\nonumber\\
\hphantom{N_I=}{} +(k(n-k)-\sum_{a\in I}n(a,\bar I))y^2-\sum_{a\in I}(\lambda- (n(a,I)+1)y+n(a,\bar I)y)^2,\label{e-NI}
\end{gather}
see \eqref{eq-w} for the definition of $n(a,I)$. Let $D_I$ be the divisor of zeros of the section $\psi_I(y,\lambda)$ on~$\hat E_T(\mathrm{pt})$, $I\subset [n]$, see Corollary~\ref{cor-2}. Then the normalized weight function can be understood as a {\em holomorphic} section of an admissible bundle:
\begin{gather*}
c ^*w^+_I\in\Gamma(X_{k,n}, p_T^*\mathcal L_I(D_I)\otimes\mathcal T_{k,n}).
\end{gather*}
Here the notation $\mathcal L(D)$ means as usual the invertible sheaf of meromorphic sections of a~sheaf~$L$ whose poles are bounded by the divisor $D$.
\begin{Example} Let $n=1$. Then $w^+_\varnothing(t,z,y,\lambda)=1$ and $c ^*w_\varnothing=1$ is a section of the trivial bundle ($\mathcal L_{\varnothing}$ and $\mathcal T_{0,1}$ are both trivial). For $k=1$,
\begin{gather*}
w^+_{\{1\}}(t,z,y,\lambda)=\frac{\theta(\lambda+t-z)}{\theta(\lambda)\theta(\lambda-y)},
\end{gather*}
and $c ^*w^+_{\{1\}}$ is obtained by substituting $t=z$:
\begin{gather*}
c ^*w^+_{\{1\}}(z,y,\lambda)=\frac1{\theta(\lambda-y)}.
\end{gather*}
This is a meromorphic section of the line bundle $\mathcal L\big({-}(\lambda-y)^2,0\big)$ with a~simple pole at $\lambda=y$ on $\hat E_T(X_{1,1})=\hat E_T(\mathrm{pt})\cong E^3$. The quadratic form is composed from the quadratic forms $2z\lambda$ of~$\mathcal T_{1,1}$ and $\mathcal -2z\lambda-(\lambda-y)^2$ of~$\mathcal L_{\{1\}}$.
\end{Example}
\subsection{Elliptic cohomology classes and stable envelope}\label{ss-5.5}
Here we introduce an elliptic version of the Maulik--Okounkov {\em stable envelope}. It is constructed in terms of weight functions. In Appendix~\ref{sec:appendix} we give an axiomatic definition in the spirit of~\cite{MO}. It would be interesting to understand the relation of our definition with the one sketched in~\cite{AO}.

\begin{Definition} Let $\mathcal L\in\mathrm{Pic}\big(\hat E_T(\mathrm{pt})\big)$. A {\em $T$-equivariant elliptic cohomology class} on $X_{k,n}$ of degree $\mathcal L$ is a holomorphic section of the admissible bundle $p_T^*\mathcal L\otimes \mathcal T_{k,n}$ on $\hat E_T(X_{k,n})$. We denote by $ H_T^{\mathrm{ell}}(X_{k,n})_{\mathcal L}$ the vector space of $T$-equivariant elliptic cohomology classes of degree $\mathcal L$ on $X_{k,n}$. We denote by $H_T^{\mathrm{ell}}(X_n)_{\mathcal L}$ the $\mathfrak h$-module $\oplus_{k=0}^n H_T^{\mathrm{ell}}(X_{k,n})_{\mathcal L}$, with $k$-th direct summand of $\mathfrak h$-weight $-n+2k$.
\end{Definition}
\begin{Definition} The {\em stable envelope} is the map
\begin{gather}\label{e-stable1}
 \operatorname{Stab}\colon \ \big(\mathbb{C}^2\big)^{\otimes n}\to \oplus_{k=0}^{n}\oplus_{I\subset [n],\,|I|=k} H_T^{\mathrm{ell}}(X_{k,n})_{\mathcal L_I(D_I)},
\end{gather}
sending $v_I$ to the cohomology class $c ^*w_I^+$.
\end{Definition}
\begin{Remark} The basis vector $v_I$ should be viewed as the generator of the space of elliptic cohomology classes of the fixed point $x_I$, see Section~\ref{s-7} below.
\end{Remark}
\begin{Remark} The class $c ^*w_I^+$ has analogs in equivariant cohomology and equivariant $K$-theory of $X_{k,n}$, see \cite{GRTV,RTV,RTV2}. The analog of $c ^*w_I^+$ in equivariant cohomology is the equivariant Chern--Schwartz--MacPherson class of the open Schubert variety~$\Omega_I$, see~\cite{RV1}. Hence $c ^*w_I^+$ may be considered as an elliptic equivariant version of the Chern--Schwartz--MacPherson class.
\end{Remark}

\subsection{Sheaf of elliptic cohomology classes and theta functions}
Here we realize elliptic cohomology classes as sections of coherent sheaves on $\hat E_T(\mathrm{pt})$ and relate their sections to the theta functions with vanishing condition of Section~\ref{ss-3.3}.
\begin{Definition} Let $k=0,\dots,n$, $\mu=-n+2k$ and $\mathcal T_{k,n}$ be the twisting line bundle of Definition~\ref{def-tlb}. The {\em sheaf of elliptic cohomology classes of weight~$\mu$} is the sheaf
 \begin{gather*}
 \mathcal H_T^{\mathrm{ell}}(X_{k,n})=p_{T*}\mathcal T_{k,n}
 \end{gather*}
 on $\hat E_T(\mathrm{pt})$. Here $p_{T*}=(p_T)_*$ denotes the direct image by the structure map $p_T\colon \hat E_T(X_{k,n})\to \hat E_T(\mathrm{pt})$.
\end{Definition}
By the projection formula, $\mathcal L\otimes p_{T*}\mathcal T_{k,n}\cong p_{T*}(p_T^*\mathcal L\otimes \mathcal T_{k,n})$ for any line bundle $\mathcal L\in \mathrm{Pic}\big(\hat E_T(\mathrm{pt})\big)$. Thus a section of $\mathcal H_T^{\mathrm{ell}}(X_{k,n})\otimes \mathcal L$ on an open set $U$ is a section of the admissible line bundle $p_T^*\mathcal L\otimes \mathcal T_{k,n}$ on $p_T^{-1}(U)$. In particular,
\begin{gather*}
 H^{\mathrm{ell}}_T(X_{k,n})_{\mathcal L}=\Gamma\big(\hat E_T(\mathrm{pt}),\mathcal H^{\mathrm{ell}}_T(X_{k,n})\otimes\mathcal L\big).
\end{gather*}

The space $\Theta^+_k(z,y,\lambda)$ of theta functions introduced in Section \ref{ss-3.1} is the fiber of a vector bundle $\Theta^+_{k,n}$ on $\hat E_T(\mathrm{pt})$. In the language of Section~\ref{ss-alb1}, $\Theta^+_{k,n}=p_*\mathcal L(N^\Theta_{k,n},0)$ is the direct image by the projection $p\colon E^{(k)}\times\hat E_T(\mathrm{pt})\to \hat E_T(\mathrm{pt})$ onto the second factor of the line bundle associated with the quadratic form
\begin{gather}\label{e-NTheta}
 N^\Theta_{k,n} =2\sum_{i=1}^kt_i(\lambda+(n-k)y)+\sum_{i=1}^k\sum_{a=1}^n(t_i-z_a)^2+k(k-1)y^2.
\end{gather}
Here, as usual, the $t_i$ are coordinates on the universal covering of $E^k$ and $z_a$, $y$, $\lambda$ are coordinates on the universal covering of $\hat E_T(\mathrm{pt})$. In fact only the terms involving $t_i$ in $N^\Theta_{k,n}$ are determined by the transformation properties of the fibers $\Theta^+_k(z,y,\lambda)$. We choose the remaining terms to simplify the formulation of Theorem~\ref{th-0} below.

The space of theta functions $\bar\Theta^+_k(z,y,\lambda)$ satisfying the vanishing condition of Section~\ref{ss-3.3} is the generic fiber of a coherent subsheaf $\bar\Theta^+_{k,n}$ of $\Theta^+_{k,n}$ on $\hat E_T(\mathrm{pt})$ (it is the intersection of kernels of morphisms $\mathrm{ev}_{z_a,2}$ of coherent sheaves). The sheaves $\bar\Theta^+_{k,n}$ and $\mathcal H_T^{\mathrm{ell}}(X_{k,n})$ are closely related: there is a morphism
\begin{gather*}
 \varphi\colon \ \bar\Theta^+_{k,n}\to \mathcal H_T^{\mathrm{ell}}(X_{k,n}),
\end{gather*}
defined as follows. A section of $\bar\Theta^+_{k,n}$ on an open set $U$ is given by a~function $f(t;z,y,\lambda)$ on $\mathbb C^k\times U$, which, as a function of~$t$ belongs to $\Theta^+_k(z,y,\lambda)$ and obeys the vanishing conditions
\begin{gather*}
 f(z_a,z_a-y,t_3,\dots,t_k;z,y,\lambda)=0, \qquad a=1,\dots,n.
\end{gather*}
The morphism $\varphi$ sends $f$ to $(\varphi_If)_{I\subset [n],|I|=k}$ where $\varphi_If$ is the restriction of $\varphi f$ to $\iota_I \hat E_T(\mathrm{pt})\cong\hat E_T(\mathrm{pt})$:
\begin{gather}\label{e-phi}
 \varphi_If(z,y,\lambda)= \frac{f(t;z,y,\lambda)}{\prod\limits_{i\neq j} \theta(t_i-t_j+y)}\bigg|_{t=z_I},
\end{gather}
cf.~Section \ref{ss-3.9}.
\begin{Theorem}\label{th-0} Let $D\subset \hat E_T(\mathrm{pt})$ be the union of the hypertori $z_a=z_b+y$, $1\leq a\neq b\leq n$ and $\lambda=jy$, $-n\leq j\leq n$. The map $\varphi\colon f\mapsto (\varphi_If)_{I\subset [n], |I|=k}$ given by formula~\eqref{e-phi} is a well-defined injective morphism of $\mathcal O_{\hat E_{T}(\mathrm{pt})}$-modules
 \begin{gather*}
 \varphi\colon \ \bar\Theta^+_{k,n}\hookrightarrow \mathcal H^{\mathrm{ell}}_T(X_{k,n}),
 \end{gather*}
 which is an isomorphism on $\hat E_T(\mathrm{pt})\smallsetminus D$.
\end{Theorem}

\begin{proof} We first prove that the morphism is well-defined. The function $\varphi_If$ of $z$, $y$, $\lambda$ defines a~section of the line bundle $\mathcal L(Q,0)$ with quadratic form $Q=\Big(N^\Theta_{k,n}-\sum\limits_{i\neq j}(t_i-t_j+y)^2\Big)\big|_{t=z_I}$. An explicit calculation shows that
 \begin{gather*}
 Q=N_{k,n}|_{t=z_I,\, s=z_{\bar I}}.
 \end{gather*}
 It follows that $\varphi f$ is a meromorphic section of $p_{T*}\mathcal T_{k,n}$. By Lemma~\ref{lemma-4} (which applies to any symmetric theta function obeying the vanishing condition), $\varphi f$ is actually holomorphic.

To show that the morphism is injective, we use the fact that the weight functions $\omega_I^+$ form a~basis of $\bar\Theta^+_k(z,y,\lambda)$ at the generic point of $\hat E_T(\mathrm{pt})$, see Proposition~\ref{p-basis}. Thus every local section of $\bar\Theta^+_{k,n}$ can be written as linear combination of normalized weight functions with meromorphic coefficients. If this section is in the kernel of our morphism then its restriction to each component vanishes. By the triangularity property of weight functions of Lemma~\ref{lemma-5} all coefficients must vanish and the kernel is trivial.

 We now construct the inverse map on the complement of $D$. A section $s$ of $\mathcal H^{\mathrm{ell}}_T(X_{k,n})$ on an open set $U$ is a collection of sections $s_I$ of $\mathcal T_{k,n}$ on the various components of $p_T^{-1}(U)$ and agreeing on intersections. Then $f=\varphi^{-1}s$ is
 \begin{gather*}
 f(t;z,y,\lambda)= \sum_{I,K} \psi_K(y,\lambda)^{-1} \frac{w_K^-(z_I,z,y,\lambda)s_I(z,y,\lambda)} {\prod\limits_{a\in I,\,b\in\bar I}\theta(z_a-z_b)\theta(z_a-z_b+y)} \omega_K^+(t,z,y,\lambda).
 \end{gather*}
 It is easy to check that this is a meromorphic section of $\bar\Theta^+_{k,n}$ on $U$ with poles at $z_b=z_a+y$, $1\leq a\neq b\leq n$ and at the zeros of $\psi_K$. It is regular at the apparent poles at $z_a=z_b$ since the sections $s_I$ agree on intersections of the components. Let us compute $\varphi f$:
 \begin{gather*}
 \frac{f(t;z,y,\lambda)} {\prod\limits_{i\neq j}\theta(t_i-t_j+y)} =\sum_{I,K}\frac{w_K^-(z_I,z,y,\lambda)s_I(z,y,\lambda)} {\prod\limits_{a\in I,\,b\in\bar I}\theta(z_a-z_b)\theta(z_a-z_b+y)} w_K^+(t,z,y,\lambda).
 \end{gather*}
 The orthogonality relations, see Corollary \ref{p-ortho2}, imply
 \begin{gather*}
 \varphi_Jf =\sum_{I,K} \frac{w_K^-(z_I,z,y,\lambda) s_I(z,y,\lambda)} {\prod\limits_{a\in I,\, b\in\bar I}\theta(z_a-z_b) \theta(z_a-z_b+y)}w_K^+(z_J,z,y,\lambda) =s_J(z,y,\lambda).\tag*{\qed}
 \end{gather*}\renewcommand{\qed}{}
\end{proof}

\subsection[Symmetric group and $G$-equivariant cohomology classes]{Symmetric group and $\boldsymbol{G}$-equivariant cohomology classes}

The symmetric group $S_n$ on $n$ letters acts on $\mathbb C^n$ by permutation of coordinates. This action induces an action of $S_n$ on the Grassmannians, their cotangent bundles $X_{k,n}$ and on $T$ so that the action map
$T\times X_{k,n}\to X_{k,n}$ is $S_n$-equivariant. The induced action on the cohomology schemes $\hat E_T(X_{k,n})$ can be easily described: on $\hat E_T(\mathrm{pt})=E^n\times E^2$, $S_n$ acts by permutations of the first $n$ factors and $\sigma\in S_n$ sends the component $\iota_K\hat E_T(\mathrm{pt})$ of $\hat E_T(X_{k,n})$ to $\iota_{\sigma(K)}\hat E_T(\mathrm{pt})$ so that the diagram
\begin{gather*}
 \begin{array}{@{}c} \hat E_T(\mathrm{pt})
 \stackrel{\sigma}\longrightarrow \hat E_T(\mathrm{pt})\\
 \iota_K\searrow \quad\swarrow\iota_{\sigma(K)}\\ \hat E_T(X_{k,n})
 \end{array}
\end{gather*}
commutes for any $K\subset[n]$ with $|K|=k$ elements. The structure map $\hat E_T(X_{k,n})\to \hat E_T(\mathrm{pt})$ is $S_n$-equivariant and the quotient by the action of $S_n$ is the $G$-equivariant elliptic cohomology scheme.

\begin{Lemma}\label{lemma-8} The twisting bundle is $S_n$-equivariant, i.e., the $S_n$-action lifts to an $S_n$ action on the bundle.
\end{Lemma}
\begin{proof} This follows since the twisting bundle is the pull-back by an $S_n$-equivariant map of a~bundle on $E^{(k)}\times E^{(n-k)}$ on which the action of the symmetric group is trivial.
\end{proof}

In particular for each $\sigma\in S_n$ and admissible line bundle $\mathcal M$, we have an admissible line bundle~$\sigma^*\mathcal M$ and a map
\begin{gather*} \sigma^*\colon \ \Gamma_{\mathrm{mer}}\big(\hat E_T(X_{k,n}),\mathcal M\big)\to \Gamma_{\mathrm{mer}}\big(\hat E_T(X_{k,n}),\sigma^*\mathcal M\big),
\end{gather*}
and also a map
\begin{gather*}
 \sigma^*\colon \ \mathcal H^{\mathrm{ell}}_T(X_{k,n})\to \sigma^*\mathcal H^{\mathrm{ell}}_T(X_{k,n}).
\end{gather*}
Let $\pi\colon \hat E_T(\mathrm{pt})\to \hat E_G(\mathrm{pt})=\hat E_T(\mathrm{pt})/S_n$. Then we obtain an action of the symmetric group on $\pi_*\mathcal H^{\mathrm{ell}}_T(X_{k,n})$.
\begin{Definition}\label{def-G-equiv} Let $G=U(n)\times U(1)$. The {\em sheaf of $G$-equivariant elliptic cohomology classes} is
 \begin{gather*}
 \mathcal H^{\mathrm{ell}}_G(X_{k,n})=\pi_*\mathcal H^{\mathrm{ell}}_T(X_{k,n})^{S_n},
 \end{gather*}
a coherent sheaf on $\hat E_G(\mathrm{pt})=E^{(n)}\times E\times E$. Let $\mathcal L\in\mathrm{Pic}\big(\hat E_G(\mathrm{pt})\big)$. The space of $G$-equivariant elliptic cohomology classes of degree $\mathcal L$ on $X_{k,n}$ is $H^{\mathrm{ell}}_G(X_{k,n})_{\mathcal L}=\Gamma\big(\hat E_G(\mathrm{pt}),\mathcal H_G^{\mathrm{ell}}(X_{k,n})\otimes\mathcal L\big)$. We set $H^{\mathrm{ell}}_G(X_n)_{\mathcal L}$ to be the $\mathfrak h$-module$\oplus_{k=0}^n H^{\mathrm{ell}}_G(X_{k,n})_{\mathcal L}$, with the summandlabeled by $k$ of $\mathfrak h$-weight $-n+2k$.
\end{Definition}

\subsection{Admissible difference operators}
Recall that $\hat E_T(X_{k,n})$ has a factor $E\times E$ corresponding to the $U(1)$-action on the cotangent spaces and the dynamical parameter. For $j\in\mathbb Z$, let $\tau_j=\tau_1^j$ be the automorphism of $E\times E$ such that
\begin{gather*} \tau_j(y,\lambda)=(y,\lambda+jy).
\end{gather*} Denote also by $\tau_j$ the automorphism $\mathrm{id}\times \tau_j$ of $\hat E_T(X_{k,n})=E_A(\mathrm{Gr}(k,n))\times E\times E$. If $\mathcal L$ is a line bundle on $\hat E_T(X_{k,n})$ then $\tau_j$ lifts to a (tautological) bundle map $\mathcal L\to \tau_j^*\mathcal L$, also denoted by $\tau_j^*$. It maps meromorphic sections to meromorphic sections and is thus a~well-defined operator
\begin{gather*}
 \tau^*_j\colon \ \Gamma_{\mathrm{mer}}(\hat E_T(X_{k,n}),\mathcal L)\to\Gamma_{\mathrm{mer}}(\hat E_T(X_{k,n}),\tau^*_j\mathcal L).
\end{gather*}
\begin{Definition}\label{def-adm1} Let $k=0,\dots,n$ and $\mathcal L$ be a line bundle on $\hat E_T(\mathrm{pt})$, $\mu\in 2\mathbb Z$, $\nu\in\mathbb Z$. An {\em admissible difference operator} on meromorphic sections of an
admissible line bundle $\mathcal M_1=p_T^*\mathcal L_1\otimes \mathcal T_{k,n}$ on $\hat E_T(X_{k,n})$ of degree $(\mathcal L,\mu,\nu)$ is a~linear map $\varphi\colon\Gamma_{\mathrm{mer}}\big(\hat E_T(X_{k,n}), \mathcal M_1\big)\to \Gamma_{\mathrm{mer}}\big(\hat E_T(X_{k+\mu,n}),\mathcal M_2)\big)$ such that
 \begin{itemize}\itemsep=0pt
 \item[(i)] $\mathcal M_2$ is the admissible bundle $p_T^*\mathcal L_2\otimes \mathcal T_{k+\mu/2,n}$ with $\mathcal L_2=\mathcal L\otimes \tau_\nu^*\mathcal L_1$.
 \item[(ii)] For each section $s$ of $\mathcal M_1$ and fixed point
 $x_K\in X_{k+\mu,n}$,
 \begin{gather}\label{e-do}
 \iota_K^*\varphi(s)=\sum_{K'}\varphi_{K,K'}\tau_{\nu}^*\iota_{K'}^*s
 \end{gather}
 for some sections $\varphi_{K,K'}\in\Gamma_{\mathrm{mer}}\big(\hat E_T(\mathrm{pt}),\iota^*_K\mathcal M_2\otimes\iota^*_{K'}\tau^*_\nu \mathcal M_1^{-1}\big)$.
 \end{itemize}
\end{Definition}
By inserting the definition, we see that the line bundle of which $\varphi_{K,K'}$ is a section is
\begin{gather*}
 \iota_K^*\mathcal M_2\otimes \tau_\nu^*\iota_{K'}^*\mathcal M_1^{-1}= \mathcal L\otimes \iota^*_{K}\mathcal T_{k+\mu/2,n}\otimes \tau_\nu^*\iota_{K'}^*\mathcal T_{k,n}^{-1}.
\end{gather*}
This line bundle is independent of the admissible line bundle the operator acts on. It thus makes sense to let the same admissible difference operator act on sections of different admissible line bundles. We set
\begin{gather*}
 \mathcal A_{k,n}(\mathcal L,\mu,\nu), \qquad k=0,\dots,n,\qquad 0\leq k+\mu\leq n,
\end{gather*}
to be the space of admissible difference operators of degree $(\mathcal L,\mu,\nu)$.

It is convenient to extend the above definitions to the case of varying $k$. We denote by $X_n=\sqcup_{k=0}^n X_{k,n}$ the disjoint union of cotangent bundles to Grassmannians of subspaces of all dimensions in $\mathbb C^n$. The extended elliptic cohomology scheme
is then
\begin{gather*}
 \hat E_T(X_n)=\sqcup_{k=0}^n \hat E_T(X_{k,n}).
\end{gather*}
It comes with a map $p_T\colon \hat E_T(X_n)\to \sqcup_{k=0}^n\hat E_T(\mathrm{pt})$.

\begin{Definition} An {\em admissible line bundle} on $\hat E_T(X_n)$ is a line bundle whose restriction to each $\hat E_T(X_{k,n})$ is admissible. Let $\mathcal L=\big(\mathcal L^0,\dots,\mathcal L^n\big)$ be a line bundle on $\sqcup_{k=0}^n\hat E_T(\mathrm{pt})$, $\mu,\nu\in\mathbb Z$. An {\em admissible difference operator} of degree $(\mathcal L,\mu,\nu)$ acting on sections of an admissible line bundle~$\mathcal M_1$ is a linear map $\Gamma_{\mathrm{mer}}(\hat E_T(X_n),\mathcal M_1)\to \Gamma_{\mathrm{mer}}(\hat E_T(X_n),\mathcal M_2)$ restricting for each $k=0,\dots,n$ such that $k+\mu\in\{0,\dots,n\}$ to an admissible difference operator
 \begin{gather*}
 \Gamma_{\mathrm{mer}}\big(\hat E_T(X_{k,n}),\mathcal M_1|_{X_{k,n}}\big)\to\Gamma_{\mathrm{mer}}\big(\hat E_T(X_{k+\mu,n}),\mathcal M_2|_{X_{k+\mu,n}}\big),
 \end{gather*}
 of degree $(\mathcal L^k,\mu,\nu)$. We denote by
 \begin{gather*}
 \mathcal A_n(\mathcal L,\mu,\nu)=\oplus_{0\leq k,k+\mu\leq n} A_{k,n}\big(\mathcal L^k,\mu,\nu\big)
 \end{gather*}
 the space of admissible difference operators of degree $(\mathcal L,\mu,\nu)$.
\end{Definition}
\begin{Remark} We will not need to consider operators on components for $k$ such that $k+\mu\not\in \{0,\dots,n\}$. However to have a~correct definition we may set $\hat E_T(X_{k,n})$ to be the empty set if $k\not\in\{0,\dots,n\}$ and declare the space of sections of any line bundle on the empty set to be the zero vector space.
\end{Remark}

\subsection{Left and right moment maps}

Examples of admissible difference operators are multiplication operators by sections of pull-backs of line bundles on $\hat E_T(\mathrm{pt})$. A subclass of these operators appear as coefficients in the defining relations of
the quantum group: they are the entries of the $R$-matrix and are functions of the dynamical and deformation parameter, and appear in the relations in two different guises: with and without ``dynamical shift''. We borrow the terminology of \cite[Section~3]{EtingofVarchenko}, where these two appearances are called the left and right moment maps.

Let $\mathcal L$ be a line bundle on $\hat E_T(\mathrm{pt})$. We define two line bundles $\mu_\ell\mathcal L$, $\mu_r\mathcal L$ on $\sqcup_{k=0}^n\hat E_T(\mathrm{pt})$:
\begin{itemize}\itemsep=0pt
\item $\mu_r\mathcal L$ is the line bundle $(\mathcal L,\dots,\mathcal L)$;
\item $\mu_\ell\mathcal L$ is the line bundle $(\tau_n^*\mathcal L,\tau_{n-2}^*\mathcal L,\dots,\tau_{-n}^*\mathcal L)$.
\end{itemize}
\begin{Definition} The {\em left moment map} is the map
 \begin{gather*}
 \mu_\ell\colon \ \Gamma_{\mathrm{mer}}(\hat E_T(\mathrm{pt}),\mathcal L)\to\mathcal A_n(\mu_\ell\mathcal L,0,0),
 \end{gather*}
 sending a section $s$ to the operator whose restriction to $\mathcal A_{k,n}$ is the multiplication by $\tau_{n-2k}^*p_T^*s$.

The {\em right moment map} is the map $\mu_r\colon\Gamma_{\mathrm{mer}}\big(\hat E_T(\mathrm{pt}),\mathcal L\big) \to\mathcal A_n(\mu_r\mathcal L,0,0)$ sending $s$ to the operator whose component in $\mathcal A_{k,n}$ is
the multiplication by $p_T^*s$.
\end{Definition}
\subsection{Sections of admissible bundles as a module over multiplication operators}

Let $\mathcal L\!\in\!\mathrm{Pic}\big(\hat E_T(\mathrm{pt})\big)$ and set $\mathcal A^0_n(\mathcal L)=(\mathcal A_n(\mathcal L,0,0))$. Then the family $\mathcal A_n^0=\big(\mathcal A^0_n(\mathcal L)\big)_{\mathcal L \in \mathrm{Pic}(\hat E_T(\mathrm{pt}))}$ is a commutative subalgebra graded by the Picard group of the base. It acts on meromorphic sections of admissible bundles by mapping $\Gamma_{\mathrm{mer}}(X_{k,n},\mathcal M)$ to $\Gamma_{\mathrm{mer}}(X_{k,n},\mathcal M\otimes p_T^*\mathcal L)$ for any admissible line bundle $\mathcal M$. Then the weight functions form a system of free generators of the module of sections of admissible line bundles over $\mathcal A^0_n$ in the following sense.
\begin{Theorem}\label{th-1} Let $\mathcal L\in\mathrm{Pic}\big(\hat E_T(\mathrm{pt})\big)$. Every section $\omega\in \Gamma_{\mathrm{mer}}\big(\hat E_T(X_{k,n}),p_T^*\mathcal L\otimes\mathcal T_{k,n}\big)$ can be uniquely written as
 \begin{gather*}
 \omega=\sum_{I\subset [n],\,|I|=k} a_I \operatorname{Stab}(v_I),
 \end{gather*}
 for some $a_I\in\mathcal A^0_n\big(\mathcal L\otimes\mathcal L_I^{-1}\big)$, where $L_I$ is the line bundle of Proposition~{\rm \ref{prop-8}}.
\end{Theorem}
\begin{proof} Denote by $Y_I=\iota_I\hat E_T(\mathrm{pt})$ the component labeled by $I$. Suppose that $\omega$ is a meromorphic section vanishing on $Y_J$ for all $J>I$ and such that $\omega|_{Y_I}\neq 0$. By Lemma~\ref{lemma-5}, we can subtract from $\omega$ a multiple of $c ^*w_I^+$ to get a section that vanishes on $Y_J$, $J\geq I$. By induction we may subtract from $\omega$ a suitable linear combination of weight functions to get~$0$.
\end{proof}

\subsection[$S_n$-equivariant admissible difference operators]{$\boldsymbol{S_n}$-equivariant admissible difference operators}
\begin{Definition} An admissible difference operator is called {\em $S_n$-equivariant} if it commutes with the action of the symmetric group on sections.
\end{Definition}

\begin{Lemma}\label{l-equiv} Let $\mathcal L$ be an $S_n$-equivariant line bundle on $\hat E_T(\mathrm{pt})$. An admissible difference operator $\varphi$ of degree $(\mathcal L,\mu,\nu)$ is $S_n$-equivariant if and only if its matrix elements $\varphi_{K,K'}$ obey
 \begin{gather*}
 \sigma^*\varphi_{\sigma(K),\sigma(K')}=\varphi_{K,K'}.
 \end{gather*}
\end{Lemma}

\subsection{Graded algebras, graded modules}
Let $Q$ be a group. Recall that an $Q$-graded algebra $A$ over $\mathbb C$ is a collection $(A_\gamma)_{\gamma\in Q}$ of complex vector spaces with associative linear multiplication maps $A_\gamma\otimes A_{\gamma'}\to\mathcal A_{\gamma\gamma'}$, $a\otimes b\mapsto a\cdot b$. Let~$P$ be a set with a left action of $Q$. A $P$-graded (left) module over $A$ is a collection $(M_p)_{p\in P}$ of complex vector spaces indexed by $P$ together with linear action maps $A_\gamma\otimes M_p\to A_{\gamma\cdot p}$, $a\otimes m\mapsto a\cdot m$, obeying $(a\cdot b)\cdot m=a\cdot(b\cdot m)$. A unital $Q$-graded algebra is an $Q$-graded algebra with an identity element $1\in A_e$ in the component indexed by the identity element $e$ of $Q$. We require $1$ to act as the identity on $P$-graded modules.

\subsection{The grading of admissible difference operators}
Let $Q=\mathrm{Pic}\big(\hat E_T(\mathrm{pt})\big)\ltimes 2\mathbb Z\times \mathbb Z$ be the product of the Picard group of $\hat E_T(\mathrm{pt})\cong E^{n+2}$ by $2\mathbb Z\times \mathbb Z$ with group law
\begin{gather*}
 (\mathcal L,\mu,\nu)(\mathcal L',\mu',\nu') =(\mathcal L\otimes \tau^*_\nu \mathcal L', \mu+\mu',\nu+\nu').
\end{gather*}
\begin{Proposition} The collection $(\mathcal A_n(\mathcal L,\mu,\nu))_{(\mathcal L,\mu,\nu)\in Q}$ with the composition of operators is a~unital $Q$-graded algebra.
\end{Proposition}
\begin{proof} An admissible difference operator of degree $(\mathcal L',\mu',\nu')$ sends a section of an admissible line bundle $\mathcal M_1=p_T^*\mathcal L_1\otimes\mathcal T_{k,n}$ to a section of $\mathcal M_2=p_T^*\mathcal L_2\otimes\mathcal T_{k+\mu'/2,n}$ with $\mathcal L_2=\mathcal L'\otimes\tau^*_{\nu'}\mathcal L_1$. An operator of degree $(\mathcal L,\mu,\nu)$ sends this section to a section of $p_T^*\mathcal L_3\otimes T_{k+\mu/2+\mu'/2,n}$ with
 \begin{gather*}
 \mathcal L_3=\mathcal L\otimes \tau^*_{\nu}\mathcal L_2 =\mathcal L\otimes \tau_{\nu}^*\mathcal L'\otimes\tau^*_{\nu+\nu'}\mathcal L_1.
 \end{gather*}
It is clear that the $\mu$-components of the degree add, so the composition has degree $(\mathcal L\otimes\tau_\nu^*\mathcal L',\mu+\mu'$, $\nu+\nu')$, as required. The identity element is the multiplication by constant function~1, a~section of the trivial bundle~$\mathcal O$.
\end{proof}

\begin{Remark} There is a slight abuse of notation, since $\mathcal A(\mathcal L,\mu,\nu)$ is defined for a line bundle~$\mathcal L$ and not for its equivalence class. The point is that $\mathcal A(\mathcal L,\mu,\nu)$ for equivalent bundles $\mathcal L$ are canonically isomorphic: if $\varphi$ is an admissible difference operator of degree $(\mathcal L,\mu,\nu)$ and $\psi\colon \mathcal L\to \mathcal L'$ is an isomorphism then $\varphi'=\psi\circ\varphi\circ\psi^{-1}$ is an difference operator of degree~$(\mathcal L',\mu,\nu)$. This establishes the isomorphism
 \begin{gather*}
 \bar\psi\colon \ \mathcal A(\mathcal L,\mu,\nu)\to \mathcal A(\mathcal L',\mu,\nu),
 \end{gather*}
 which we claim is independent of $\psi$. Indeed any two choices of $\psi$ differ by the composition with an automorphism of~$\mathcal L$. Since $\operatorname{Aut}(\mathcal L)=\mathbb C^\times$, $\psi$ and $\psi'$ differ by multiplication by a nonzero scalar which does not affect $\bar \psi$.
\end{Remark}
Let $P$ be the set of pairs $(\mathcal L,\mu)$ with $\mathcal L\in\mathrm{Pic}(\hat E_T(X_n))$ and $\mu\in\mathbb Z$. Then $Q$ acts on $P$ via
\begin{gather*}
 (\mathcal L,\mu,\nu)\cdot(\mathcal L',\mu') = (\mathcal L\otimes\tau_\nu^*\mathcal L',\mu+\mu').
\end{gather*}
Let $\mathcal M=p_T^*\mathcal L_1\otimes \mathcal T_{k,n}$ be an admissible bundle on $X_{k,n}$.

Admissible difference operators map sections of admissible line bundles to sections of admissible line bundles. This is formalized as follows.
\begin{Proposition} The collection of vector spaces $\Gamma_{\mathrm{mer}}\big(\hat E_T(X_{k,n}),p_T^*\mathcal L\otimes T_{k,n}\big)$, labeled by $(\mathcal L,\mu)$, with $\mu=-n+2k$ is a~$P$-graded module over the $Q$-graded unital algebra $\mathcal A_n$ of admissible difference operators.
\end{Proposition}
\begin{proof} This is an immediate consequence of the definition, see Definition~\ref{def-adm1}(i).
\end{proof}

\section{Action of the elliptic dynamical quantum group}
In this section we construct an action of the elliptic dynamical quantum group associated with~$\mathfrak{gl}_2$ on the extended equivariant elliptic cohomology $\hat E_T(X_n)$ of the union of cotangent bundles of the
Grassmannians of planes in~$\mathbb C^n$. The action is by $S_n$-equivariant admissible difference operators acting on admissible line bundles on the cohomology scheme. Thus each generator~$L_{ij}(w)$ \mbox{($i,j\in\{1,2\})$} of the elliptic dynamical quantum group acts on sections of any admissible line bundle by an admissible difference operator of some degree $(\mathcal L_{ij}(w),\mu_{ij},\nu_{ij})$ which we give below. We also compute the action on $T$-equivariant elliptic cohomology classes and use the $S_n$-equivariance to show that the action descends to an action on $G$-equivariant classes, with $G=U(n)\times U(1)$.

We construct the action in such a way that at the generic fibre of $\hat E_T(X_{k,n})\to E_T(\mathrm{pt})=E^n\times E$ (i.e., for fixed $z_1,\dots,z_n,y$) the map \eqref{e-stable1} defines a morphism of representations from the tensor product of evaluation representations. In other words, suppose that
\begin{gather*}
 L_{ij}(w)v_I=\sum_K L_{ij}(w,z,y,\lambda)_I^Kv_K,
\end{gather*}
for some meromorphic coefficients $L_{ij}(w,z,y,\lambda)_I^K$. Then we want that
\begin{gather}\label{e-action}
 L_{ij}(w)\operatorname{Stab}(v_I) =\sum_K L_{ij}(w,z,y,\lambda)_I^K\operatorname{Stab}(v_K).
\end{gather}
The matrix coefficients $L_{ij}(w,z,y,\lambda)_I^K$ are certain meromorphic functions of $z,y,\lambda$ with theta function-like transformation properties and can thus be considered as meromorphic sections line bundles on $\hat E_T(\mathrm{pt})$. Therefore each summand on the right-hand side is a meromorphic section of an admissible line bundle. The content of the following theorem is that the sum defines uniquely an admissible difference operator.
\begin{Theorem}\label{th-2}\quad
\begin{enumerate}\itemsep=0pt
\item[$(i)$] The formula \eqref{e-action} uniquely defines admissible difference operators $L_{ij}(w)$, $i,j\in\{1,2\}$, of degree $(\mathcal L_{ij}(w),2(i-j),\epsilon(j))$ with $\epsilon(1)=-1$, $\epsilon(2)=1$, for some $S_n$-equivariant line bundle $\mathcal L_{ij}(w)$ on $\hat E_T(\mathrm{pt})$.
\item[$(ii)$] These operators obey the RLL relations of the elliptic dynamical quantum group in the form
 \begin{gather*}
 \mu_\ell R(w_1\!-w_2,y,\lambda)^{(12)} L(w_1)^{(13)} L(w_2)^{(23)} = L(w_2)^{(23)}L(w_1)^{(13)} \mu_rR(w_1\!-w_2,y,\lambda)^{(12)}.
 \end{gather*}
 Here the coefficients of the quadratic relations are in $\mathcal A^0_n$ and the action of $\mu_\ell$, $\mu_r$ is on each matrix element of~$R$.
\end{enumerate}
\end{Theorem}

The proof of this theorem is by explicit description of the action and is parallel to the case of Yangians and affine quantum enveloping algebras, \cite{GRTV,RTV2,RV2}. We give the formulae for the action in Sections~\ref{ss-6.1},~\ref{ss-6.2} and \ref{ss-6.3}. The proof of Theorem~\ref{th-2} is in Section~\ref{ss-6.4}.

By Theorem \ref{th-2}, the generators $L_{ij}(w)$ send meromorphic sections of admissible bundles to meromorphic sections of admissible bundles. The next result gives a more precise control on the poles of coefficients. We give the action on holomorphic sections, i.e., equivariant elliptic cohomology classes, both for the torus $T=U(1)^n\times U(1)$ and the group $G=U(n)\times U(1)$.

\begin{Theorem}\label{th-3}\quad
\begin{enumerate}\itemsep=0pt
\item[$(i)$] Let $D$ be the divisor on $\hat E_T(\mathrm{pt})$ whose components are the hypersurfaces defined by equations $z_a+y=w$, for $1\leq a\leq n$ and $\lambda+yj=0$, for $j=-n\dots,n-1,n$. Then $L_{ij}(w)$ maps $H^{\mathrm{ell}}_T(X_n)_{\mathcal L}$ to $H^{\mathrm{ell}}_T(X_n)_{\tau_{\epsilon(j)}^*\mathcal L\otimes \mathcal L_{ij}(w)(D)}$ for any $\mathcal L\in\mathrm{Pic}\big(\hat E_T(\mathrm{pt})\big)$.
\item[$(ii)$] Let $H^{\ell}_G(X_n)_{\mathcal L}$ be the space of $G$-equivariant elliptic cohomology classes of degree $\mathcal L$ for $G=U(n)\times U(1)$, see Definition~{\rm \ref{def-G-equiv}}. Let $\pi\colon \hat E_T(\mathrm{pt})\to \hat E_G(\mathrm{pt})$ be the canonical projection. View line bundles on $\hat E_G(\mathrm{pt})$ as $S_n$-equivariant line bundles on $\hat E_T(\mathrm{pt})$. Then the opera\-tors~$L_{ij}(w)$ induce well-defined operators from $H^{\mathrm{ell}}_G(X_n)_{\mathcal L}$ to $H^{\mathrm{ell}}_G(X_n)_{\tau_{\epsilon(j)}^*\mathcal L\otimes \mathcal L_{ij}(w)(\pi(D))}$ for each $\mathcal L\in\mathrm{Pic}\big(\hat E_G(\mathrm{pt})\big)$.
\end{enumerate}
\end{Theorem}
The proof of this theorem is contained in Section~\ref{ss-6.4}.
\begin{Remark} Let $q\colon\hat E_T(\mathrm{pt})=E_T(\mathrm{pt})\times E\to E_T(\mathrm{pt})$ be the projection onto the first factor. Since the action of the generators $L_{ij}(w)$ is by admissible difference operators it preserves the fiber of $q_*\mathcal H_T^{\mathrm{ell}}(X_n)$ at a generic point of the non-extended $E_T(\mathrm{pt})$. If we realize this fiber as a~certain space of functions of $\lambda$ and tensor with all meromorphic functions of $\lambda$ we get a~representation of the quantum group in the sense of Section~\ref{ss-2.1}. By construction, it is isomorphic to the tensor product of evaluation representations. Thus we can think of the action of the quantum group on equivariant elliptic cohomology classes as a tensor product of evaluation representations with variable evaluation points and deformation parameter.
\end{Remark}

\subsection{Action of the Gelfand--Zetlin subalgebra}\label{ss-6.1}
The Gelfand--Zetlin subalgebra is the commutative subalgebra generated by $L_{22}(w)$ and the determinant $\Delta(w)$. As shown in Section~\ref{ss-3.10} these operators act diagonally in the basis $\xi_I$ of $V(z_1)\otimes\cdots\otimes V(z_n)$. It follows that the vectors
\begin{gather*}
 \hat\xi_I=\sum_{|J|=k}\frac{w^-_J(z_I,z,y,\lambda)}{\prod\limits_{a\in I,\,b\in\bar I}\theta(z_a-z_b+y)} \operatorname{Stab}(v_J),
\end{gather*}
(cf.~\eqref{e-xi}), which by construction are sums of sections of certain admissible line bundles, are eigenvectors of the Gelfand--Zetlin subalgebra. It turns out that they are sections of admissible bundles with support on a single irreducible component of $\hat E_T(X_n)$:
\begin{Proposition} Let $I,K\subset [n]$, $|I|=|K|=k$. The restriction of $\hat\xi_I$ to the component of~$\hat E(X_{k,n})$ labeled by $K$ is
\begin{gather*}
\iota_K^*\hat\xi_I= \begin{cases} \prod\limits_{a\in I,\,b\in\bar I}\theta(z_a-z_b),& $I=K$,\\
0,& otherwise.
\end{cases}
\end{gather*}
\end{Proposition}
\begin{proof} From the definition of $\hat\xi_I$ and $\operatorname{Stab}(v_J)$ we have
\begin{gather*}
\iota_K^*\hat\xi_I =\iota_K^*\bigg(\sum_J \frac{ w_J^-(z_I,z,y,\lambda)}{\prod\limits_{a\in I, \,b\in \bar{I}} \theta(z_a-z_b+y)} w^+_J(t,z,y,\lambda) \bigg) \\
\hphantom{\iota_K^*\hat\xi_I}{} =\frac{1}{\prod\limits_{a\in I,\, b\in \bar{I}}\theta(z_a-z_b+y)} \sum_J w_J^-(z_I,z,y,\lambda) w^+_J(z_K,z,y,\lambda),
\end{gather*}
which, using Corollary \ref{p-ortho2}, proves our statement.
\end{proof}

Thus we can write any section $(s_I)_{I\subset[n]}$ of an admissible bundle as linear combination
\begin{gather*}
\sum \frac{s_I\hat\xi_I}{\prod\limits_{a\in I,\,b\in\bar I}\theta(z_a-z_b)}.\end{gather*}
Since the $\hat\xi_I$ are eigenvectors the action of the Gelfand--Zetlin algebra is given by admissible difference operator with diagonal matrices of coefficients.

The action of the determinant $\Delta(w)$ is easiest to describe: it is given by an admissible difference operator of degree $(\mathcal L_\Delta(w),0,0)$, where $\mathcal L_\Delta(w)=\mathcal L(N_\Delta,v_\Delta(w))$ is the line bundle associated with the data
\begin{gather*}
 N_\Delta(z,y,\lambda)= \sum_{a=1}^n (2z_a+y)y, \qquad v_\Delta(w;z,y,\lambda)=-wy.
\end{gather*}
Since $N_\Delta$ and $v_\Delta$ are symmetric under permutations of the variables $z_i$, the corresponding bundle is naturally $S_n$-equivariant. The determinant acts on sections of any admissible line bundle $\mathcal L_1$ on $\hat E_T(X_{k,n})$ it acts by multiplication by the section
\begin{gather*}
 \prod_{i=1}^n\frac{\theta(w-z_i+y)}{\theta(w-z_i)}
\end{gather*}
of $\mathcal L_\Delta(w)$.

The action of $L_{22}(w)$ is by an operator of degree $(\mathcal L_{22}(w),0,1)$. It is defined on the components by
\begin{gather*}
 \iota^*_KL_{22}(w)s=\prod_{i\in K}\frac{\theta(w-z_i)}{\theta(w-z_i-y)} \tau_1^*\iota_K^*s,
\end{gather*}
and $\mathcal L^k_{22}(w)=\mathcal L\Big(2k\sum\limits_{a=1}^nz_ay,kwy\Big)$, which is an $S_n$-equivariant line bundle.
\begin{Lemma}\label{l-L22Delta} These formulae define $S_n$-equivariant admissible difference operators $\Delta(w)\in\mathcal A(\mathcal L_\Delta(w),0,0)$ and $L_{22}(w)\in\mathcal A(\mathcal L_{22}(w),0,1)$.
\end{Lemma}
\begin{proof} Both operators are defined by diagonal matrices $(\varphi_{K,K'})$ in the notation of Definition~\ref{def-adm1}. It is straightforward to check that the diagonal matrix elements $\varphi_{K,K}$ are sections of the correct line bundle. The equivariance property of Lemma~\ref{l-equiv} is clearly satisfied. Moreover the divisor of poles does not contain any diagonal $\Delta_I=\big\{z\in E^{n+2}\colon z_i=z_j, \forall\, i,j\in I\big\}$, $I\subset[n]$ so that the difference operator maps meromorphic sections to meromorophic sections. It remains to check that the sections on the different components coincide on their intersections, namely that for every $a\in K$, $b\in\bar K$,
 \begin{gather*}
 \varphi_{K,K}|_{z_a=z_b}=\varphi_{\tilde K,\tilde K}|_{z_a=z_b}, \qquad \tilde K=K\smallsetminus \{a\}\cup\{b\}.
 \end{gather*}
 This can be checked directly but also follows from the equivariance condition for the permutation of $a$ and $b$.
\end{proof}

\subsection[Action of $L_{12}$ and $L_{21}$]{Action of $\boldsymbol{L_{12}}$ and $\boldsymbol{L_{21}}$}\label{ss-6.2}
Let $k=1,\dots,n$ and $K\subset[n]$ with $|K|=k-1$, then
\begin{gather*} \iota^*_K L_{12}(w)s =(-1)^k\theta(y)\sum_{a\in \bar K} \frac {\theta(\lambda+w-z_a+(n-2k+1)y)} {\theta(w-z_a-y)} \\
\hphantom{\iota^*_K L_{12}(w)s =}{} \times \prod_{j\in K} \frac {\theta(w-z_j)} {\theta(w-z_j-y)} \frac {\prod\limits_{j\in K} \theta(z_a-z_j-y) } { \prod\limits_{j\in \bar K\smallsetminus\{a\}} \theta(z_a-z_j) } \iota_{K\cup\{a\}}^*\tau_1^*s.
\end{gather*}
\begin{Lemma}\label{l-L12} $L_{12}(w)$ is an $S_n$-equivariant admissible difference operator of degree $(\mathcal L_{12}(w),\!{-}2,1)$ with
 \begin{gather*}
 \mathcal L^k_{12}(w)=\mathcal L\big({-}(\lambda+(n-2k)y)^2,-w(\lambda+(n-k+1)y)\big).
 \end{gather*}
\end{Lemma}
Let $k=0,\dots,n-1$ and $K\subset[n]$ with $|K|=k+1$, then
\begin{gather*}
\iota^*_K L_{21}(w)s =(-1)^{n-k} \frac {\theta(y)}{\theta(\lambda)\theta(\lambda-y)} \sum_{a\in K} \frac{\theta(\lambda-w+z_a)} {\theta(w-z_a-y)} \\
\hphantom{\iota^*_K L_{21}(w)s =}{} \times \prod_{j\in K\smallsetminus\{a\}} \frac {\theta(w-z_j)} {\theta(w-z_j-y)} \frac {\prod\limits_{j\in \bar K} \theta(z_j-z_a-y) } {\prod\limits_{j\in K\smallsetminus\{a\}} \theta(z_j-z_a) }\iota_{K\smallsetminus\{a\}}^*\tau_{-1}^*s
\end{gather*}
\begin{Lemma}\label{l-L21} $L_{21}(w)$ is an $S_n$-equivariant admissible difference operator of degree $(\mathcal L_{21}(w),2,\!{-}1)$ with
 \begin{gather*}
 \mathcal L^k_{21}(w)=\mathcal L\left({-}\lambda^2-(n-2k+2)y^2+2y\left(\lambda-\sum_{i=1}^nz_i\right),w(\lambda-(k+1)y)\right).
 \end{gather*}
\end{Lemma}

Lemmas \ref{l-L12} and \ref{l-L21} are proved the same way as Lemma~\ref{l-L22Delta}. The only new feature is the appearance of simple poles on diagonals $z_i=z_j$ and it is thus not a priori clear that these operators map meromorphic sections to meromorphic sections in the sense of Definition~\ref{def-mero}. The point is that when acting on meromorphic sections, these poles cancel by the equivariance conditions. For example let us consider the behaviour of $\varphi=L_{12}(w)s$ in the vicinity of the diagonal $z_a=z_b$. The matrix element $\varphi_{K,K\cup\{a\}}$ has a simple pole there if $b\in \bar K\smallsetminus\{a\}$ and so has the matrix element $\varphi_{K,K'\cup\{b\}}$ which by equivariance is obtained from $\varphi_{K,K\cup\{a\}}$ by the transposition $\sigma_{12}$ of $a$, $b$. In local coordinates and trivializations compatible with the $S_n$-action, by setting $f_j=\iota_{K\cup\{j\}}\tau^*_1s$, $j=a,b$, the potentially singular term in $\iota_K^*L_{12}(w)s$ at $z_a=z_b$ has the form
\begin{gather*}
 \frac {g(z,y,\lambda)}{z_a-z_b}f_a(z,y,\lambda) +\frac{g(s_{ab}z,y,\lambda)}{z_b-z_a}f_b(\sigma_{12}z,y,\lambda).
\end{gather*}
Since $s_{ab}z=z$ and $f_a=f_b$ on the diagonal $z_a=z_b$, the poles cancel. The same argument works for $L_{21}$.

\subsection[Action of $L_{11}$]{Action of $\boldsymbol{L_{11}}$}\label{ss-6.3}
Since $L_{22}(w)$ is an invertible admissible difference operator, the action of $L_{11}(w)$ can be obtain from the action of the Gelfand--Zetlin algebra and the action of $L_{12}$, $L_{21}$ via the formula for the determinant
\begin{gather}\label{e-determinant}
 \Delta(w)=\frac{\mu_\ell(\theta(\lambda))}{\mu_r(\theta(\lambda))} (L_{11}(w+y)L_{22}(w)-L_{21}(w+y)L_{12}(w)).
\end{gather} Here $\theta(\lambda)$ is considered as a section of the bundle $\mathcal L(N,0)$ on $E^2$ with $N(y,\lambda)=\lambda^2$.

\subsection{Proof of Theorems \ref{th-2} and \ref{th-3}}\label{ss-6.4}
Theorem \ref{th-2}(i) for $L_{22}$, $L_{12}$ and $L_{21}$ follows from Lemma~\ref{l-L22Delta}, Lemma~\ref{l-L12} and Lemma~\ref{l-L21}, respectively. The operator $L_{11}$ can be expressed as composition of these admissible difference operators via the determinant and is thus also admissible. Part~(ii) follows by construction. To prove Theorem~\ref{th-3}(i) we need to check that in the matrix elements of $L_{ij}(w)$ only simple poles at $z_i=w+y$ and $\lambda=jy$, $j\in\mathbb Z$ appear. This is clear from the explicit formulae except for~$L_{11}(w)$. To prove it for this operator, we use two formulae for it: one is using the definition and the orthogonality relations, and the other expressing it in terms of the other $L_{ij}$ and the determinant, as in Section~\ref{ss-6.3}.

The first formula gives
\begin{gather}\label{e-99}
 \iota_L^*L_{11}(w)s=\sum_{I,J,K} \frac{w^-_K(z_I)L_{11}(w)_K^Jw^+_J(z_L)} {\prod\limits_{a\in I,\,b\in\bar I}\theta(z_a-z_b) \theta(z_a-z_b+y)} \tau_{-1}^*\iota_I^*s.
\end{gather}
The matrix elements $L_{11}(w)_K^J$ of $L_{11}(w)$ in the tensor basis $v_I$ of $\big(\mathbb C^2\big)^{\otimes n}$ are sums of products of matrix elements of $R$-matrices and have at most simple poles at $z_a=w+y$ and possible poles at $\lambda=jy$, $j\in\mathbb Z$, see~\eqref{e-tensor}. Thus the right-hand side of the~\eqref{e-99} has (among other poles) at most simple poles at $z_a=w+y$. The second formula is in terms of the determinant:
\begin{gather*}
 L_{11}(w)s=\left(\frac{\theta(\lambda)} {\theta(\lambda-y\mu)}\Delta(w-y) +L_{21}(w-y)L_{12}(w)\right)L_{22}(w+y)^{-1}s,
\end{gather*}
for $s\in H_T^{\mathrm{ell}}(X_{k,n})_{\mathcal M}$ with $\mu=-n+2k$. From this formula and the explicit expression of $L_{12}$, $L_{21}$, $L_{22}$ we see that only simple poles at $\lambda=y\mu$, $-n\leq \mu\leq n$ occur and that the remaining apparent poles at $z_a-z_b=0$, $z_a-z_b+y=0$ in \eqref{e-99} are spurious.

Finally Theorem \ref{th-3}(ii) holds since the bundles $\mathcal L_{ij}(w)$ are $S_n$-equivariant (and can thus be viewed as line bundles on the quotient) and the action is given by $S_n$-equivariant difference operators.

\section{Shuffle products and stable envelopes for subgroups}\label{s-7}
The stable envelope of \cite{MO} is a map from the equivariant cohomology of the fixed point set for a torus action on a Nakajima variety to the equivariant cohomology of the variety. The goal of this section is to extend this interpretation of the stable envelope to the elliptic case for cotangent bundles of Grassmannians. In our construction the stable envelope is built out of weight functions, which in turn are obtained from shuffle products of elementary weight functions associated with the one-point spaces $T^*\mathrm{Gr}(0,1)$, $T^*\mathrm{Gr}(1,1)$. Thus the first step is to extend the fiber-by-fiber construction of the shuffle product of Section~\ref{s-3} to a shuffle product defined on sections of the coherent sheaf $\bar\Theta^+_{k,n}$ on $\hat E_T(\mathrm{pt})$. By using the isomorphism (outside the divisor~$D$) of Theorem~\ref{th-0}, we get a~shuffle product on the sections of the sheaf of elliptic cohomology classes $\mathcal H^{\mathrm{ell}}_{T}(X_{k,n})$. The $n$-fold shuffle product of classes in $\mathcal H^{\mathrm{ell}}_{T}(X_{k,1})$, $k=0,1$ is then essentially the stable envelope.

We propose to view shuffle products of factors of an arbitrary number of factors as stable envelopes corresponding to subgroups of~$T$. Their geometric interpretation is that they correspond to maps from the cohomology of the fixed point set for the action of a subgroup of the torus~$T$, cf.~\cite[Section 3.6]{MO}.

The basic case, which as we shall see corresponds to the shuffle product of two factors, is the subgroup $B_m\subset U(1)^n$
\begin{gather*}
 B=B_m=\{(\underbrace {z,\dots,z}_m,1,\dots,1)\in A \colon z\in U(1)\},
\end{gather*}
isomorphic to $U(1)$.

Fixed points for the action of this subgroup on $\mathrm{Gr}(k,n)$ are $k$-planes of the form $V_1\oplus V_2$, with~$V_1$ in the span of the first $m$ coordinate axes and~$V_2$ in the span of the last $n-m$ coordinate axes. Thus the fixed point set decomposes into connected components according to the dimension of~$V_1$. Each of these components is a product of Grassmannians. Similarly, the $B_m$-invariant part of the cotangent space at a fixed point splits as a direct sum of cotangent spaces at the factors and we get an isomorphism
\begin{gather*}
 X_{k,n}^{B_m}\cong \sqcup_{d=0}^k X_{d,m}\times X_{k-d,n-m}.
\end{gather*}
As above we consider the action of $A_n=U(1)^n$ on $X_{k,n}$. Then the embedding $X_{d,m}\times X_{k-d,n-m}\hookrightarrow X_{k,n}$ is $A_m\times A_{n-m}= A_n$-equivariant. The K\"unneth formula~\cite{GKV} predicts that this embedding induces a map
\begin{gather}\label{e-Kunn}
 E_{A_m}(X_{d,m})\times E_{A_{n-m}}(X_{k-d,n-m})\to E_{A_n}(X_{k,n}),
\end{gather}
In the description as a fiber product,
\begin{gather*}
 E_{A_m}(X_{d,m})=E^{(d)}\times E^{(m-d)}\times_{E^{(m)}}E^m,
\end{gather*}
and the map is the obvious one: $((t',s',z'),(t'',s'',z''))\mapsto (p(t',t''),p(s',s''),z',z'')$. Here $t'\in E^{(d)}$, $t''\in E^{(k-d)}$, $p\colon E^{(d)}\times E^{(k-d)}\to E^{(k)}$ is the canonical projection and similarly for the other factors.

As in Section~\ref{ss-Extended}, we consider the extended equivariant elliptic cohomology $\hat E_{T_n}(X_{k,n})=E_{A_n}(X_{k,n})\times E^2$ for the torus $T_n=A_n\times U(1)$ where the additional $U(1)$ factor acts by multiplication on each cotangent space. We then have the corresponding embedding
\begin{gather}\label{e-embedding}
 \hat E_{T_m}(X_{d,m})\times_{E^2}\hat E_{T_{n-m}}(X_{k-d,n-m})\to \hat E_{T_n}(X_{k,n}),
\end{gather}
where the map to $E^2$ is the projection onto the second factor. Both are schemes over $\hat E_{T_n}(\mathrm{pt})=\hat E_{T_m}(\mathrm{pt})\times_{E^2}E_{T_{n-m}}(\mathrm{pt})$.

\begin{Proposition} The shuffle product of Proposition~{\rm \ref{prop-2}} defines a map
\begin{gather*}
*\colon \ \tau^*_{n''-2k''}\bar\Theta^+_{k',n'}\boxtimes\bar\Theta^+_{k'',n''}\to\bar\Theta^+_{k,n}\otimes \mathcal L_{k',k'',n',n''}
\end{gather*}
of sheaves of $\mathcal O_{\hat E_T(\mathrm{pt})}$-modules, where $k=k'+k''$, $n=n'+n''$ and $\mathcal L_{k',k'',n',n''}= \mathcal L\Big(k''y\Big((n'-k')y-2\sum\limits_{a=1}^{n'}z_a\Big),0\Big)\in\mathrm{Pic}\big(\hat E_{T_n}(\mathrm{pt})\big)$.
\end{Proposition}
\begin{proof} The sheaf $\Theta^+_{k,n}$ is defined by the quadratic form $N^\Theta_{k,n}(t,z,y,\lambda)$, see~\eqref{e-NTheta}. Let us write $t=(t',t'')$, $z=(z',z'')$, with $t'=(t_1,\dots,t_{k'})$, $t''=(t_{k'+1},\dots,t_k)$ and similarly for $z$. Then $\tau^*_{n''-2k''}\bar\Theta^+_{k',n'}\boxtimes\bar\Theta^+_{k'',n''}$ is associated with the quadratic form
 \begin{gather*}
 M(t,z,y,\lambda)=N^\Theta_{k',n'}(t',z',y,\lambda+y(n''-2k'')) +N^\Theta_{k'',n''}(t'',z'',y,\lambda).
 \end{gather*}
 The shuffle product maps a section of this bundle to a section of a bundle associated with the sum of this quadratic form and the quadratic forms of the theta functions in $\varphi^+$, see Proposition~\ref{prop-0}, namely
 \begin{gather*}
 M(t,z,y,\lambda) + \sum_{j=1}^{k'}\sum_{l=k'+1}^k \big((t_j-t_l+y)^2-(t_j-t_l)^2\big)\\
\hphantom{M(t,z,y,\lambda)}{} +\sum_{l=k'+1}^k \sum_{a=1}^{n'}(t_l-z_a+y)^2+\sum_{j=1}^{k'}\sum_{b=n'+1}^n (t_j-z_b)^2.
 \end{gather*}
 It is straightforward to verify that this is equal to
 \begin{gather*}
 N^\Theta_{k,n}(t,z,y,\lambda)+k''y\left((n'-k')y-2\sum_{a=1}^{n'}z_a\right).
 \end{gather*}
 This shows that the shuffle product takes values in $\Theta^+_{k,n}\otimes\mathcal L_{k',k'',n',n''}$. The fact that it actually takes values in the subsheaf defined by the vanishing condition follows from Proposition~\ref{prop-2}.
\end{proof}

\begin{Definition} Let $k=k'+k''$, $n=n'+n''$. The {\em unnormalized stable envelope} associated with the component $X_{k',n'}\times X_{k'',n''}$ of the fixed point set $X_{k,n}^{B_{n'}}$ is the shuffle product map
 \begin{gather*}
 \widetilde{\operatorname{Stab}}\colon \ \tau_{n''-2k''}^*\bar\Theta^+_{k',n'} \boxtimes\bar\Theta^+_{k'',n''} \to \bar\Theta^+_{k,n}\otimes\mathcal L_{k',k'',n',n''}
 \end{gather*}
 of sheaves of $\mathcal O_{\hat E_{T_n}(\mathrm{pt})}$-modules.
\end{Definition}
By using the isomorphism $\bar\Theta_{k,n}\cong\mathcal H_{T_n}^{\mathrm{ell}}(X_{k,n})$ on the complement of the divisor $D$ of Theorem~\ref{th-0}, we obtain a map
\begin{gather*}
 \widetilde{\operatorname{Stab}}\colon \ \tau_{n''-2k''}^*\mathcal H^{\mathrm{ell}}_{T_{n'}}(X_{k',n'}) \boxtimes\mathcal H^{\mathrm{ell}}_{T_{n''}}(X_{k'',n''}) \to\mathcal H^{\mathrm{ell}}_{T_n}(X_{k,n})\otimes \mathcal L_{k',k'',n',n''},
\end{gather*}
on $\hat E_T(\mathrm{pt})\smallsetminus D$.

More generally, we may consider subgroups $B=B_{n_1}\times\cdots \times B_{n_r}\subset U(1)^n$ whose fixed point sets have components $X_{k_1,n_1}\times\cdots \times X_{k_r,n_r}$ and define stable envelopes given by $r$-fold shuffle products and thus by compositions of stable envelopes for two factors.

Two special cases give the stable envelope of Section~\ref{ss-5.5} and the action of the elliptic dynamical quantum group.

In the first case we take $B=U(1)^n$. The fixed points are isolated and labeled by $I\subset[n]$. We think of the fixed point labeled by $I$ as a product $X_{k_1,1}\times\cdots\times X_{k_n,1}$ with $k_i=1$ if $i\in I$ and $k_i=0$ otherwise. The unnormalized stable envelope on the component labeled by $I\subset [n]$ is then
\begin{gather*}
 \widetilde {\operatorname{Stab}}\colon \ \boxtimes_{j=1}^n \tau^*_{w(j,I)}\mathcal H^{\mathrm{ell}}_{T_1}(X_{k_i,1}) \to\mathcal H^{\mathrm{ell}}_{T_1}(X_{k_i,1})\otimes \mathcal M_I,
\end{gather*}
(the factors are ordered from left to right) for some suitable line bundle $\mathcal M_I\in\mathrm{Pic}(\hat E_{T_n}(\mathrm{pt}))$ obtained as tensor product of line bundles $L_{k',k'',n',n''}$. In this case the map is defined everywhere, not just on the complement of $D$, since $\bar\Theta_{k,1}\cong \mathcal H^\mathrm{ell}_{T_1}(X_{k,1})$ on $\hat E_{T_1}(\mathrm{pt})$.

The stable envelope of Section \ref{ss-5.5} is obtained by taking the tensor product with suitable line bundles $\tau^*_{w(j,I)}\mathcal M_{k_j}$ on $\hat E_{T_n}(\mathrm{pt})$ so that $\oplus_{k=0}^1\Gamma(X_{k,1},\mathcal
H^{\mathrm{ell}}_{T_1}(X_{k,1})\otimes \mathcal M_k)$ is identified with $\mathbb C^2$ via the basis $\omega_0^+,\omega_1^+$, passing to global sections and normalizing by dividing by $\psi_I$.

In the second case we reproduce the construction of Section~\ref{ss-3.7} in the global setting. Thus we consider the stable envelope for two factors $X_{d,1}\times X_{k-d,n}\subset X_{k,n+1}$.
We obtain two maps
\begin{gather*}
 \oplus_{d=0}^1\tau^*_{n-2(k-d)} \bar\Theta^+_{d,1}\boxtimes \bar\Theta^+_{k-d,n}\otimes \mathcal L_{d,k-d,1,n}^{-1} \to \bar\Theta^+_{k,n+1} \\
 \qquad{} \oplus_{d=0}^1\tau^*_{1-2d} \bar\Theta^+_{k-d,n} \boxtimes \bar\Theta^+_{d,1}\otimes \mathcal L_{k-d,d,n,1}^{-1} \to \bar\Theta^+_{k,n+1}
\end{gather*}
which are invertible at a generic point. Since $\bar\Theta^+_{k,1}$ is isomorphic to $\mathcal H^{\mathrm{ell}}_{T_1}(X_{k,1})$ we get a~map
\begin{gather*}
 \oplus_{d=0}^1\tau^*_{1-2d} \bar\Theta^+_{k-d,n} \boxtimes\mathcal H^{\mathrm{ell}}_{T_1}(X_{d,1}) \otimes \mathcal L_{k-d,d,n,1}^{-1}\\
\qquad{} \to \oplus_{d=0}^1\tau^*_{n-2(k-d)} \mathcal H^{\mathrm{ell}}_{T_1}(X_{d,1})\boxtimes \bar\Theta^+_{k-d,n}\otimes \mathcal L_{d,k-d,1,n}^{-1},
\end{gather*}
defined on some dense open set. This map contains the information of the action of the elliptic dynamical quantum group on the elliptic cohomology of $\hat E_{T_n}(X_{n})$. The action of the generators is given as explained in Section~\ref{ss-3.7}: one needs as above to take the tensor product with a suitable line bundle to associate elliptic cohomology classes $\omega^+_1,\omega^+_0$ with the standard basis of $\mathbb C^2$. Then we are in the setting of Section~\ref{ss-3.7} and we obtain an action of the elliptic dynamical quantum group which is up to gauge transformation the one described in the previous section.

\appendix
\section{Axiomatic definition of elliptic stable envelopes}\label{sec:appendix}

In this section we give an axiomatic definition of the elliptic stable envelopes in the spirit of Maulik--Okounkov \cite[Section 3.3]{MO}, see also \cite{RTV, RTV2, RV2}.

Recall that $c ^*w^+_I$ is a meromorphic section (with controlled denominators) of an appropriate line bundle over $\hat{E}_T(X_{k,n})$. The scheme $\hat{E}_T(X_{k,n})$ has components $Y_J=\iota_J \hat{E}_T(\mathrm{pt})$, and the restriction of a section to $Y_J$ is the result of substituting the variables $t_i$ by $z_J=(z_i)_{i\in J}$.

A meromorphic section of an admissible line bundle $p_T^*\mathcal L(N,0)\otimes \mathcal T_{k,n}$ restricted to $Y_J$ can be written as a meromorphic function $F\colon \C^{n+2} \to \C$ whose transformation properties with respect to the lattice $\Z^{n+2}+\tau \Z^{n+2}$ are determined by $p_T^*\mathcal L(N,0)\otimes \mathcal T_{k,n}$, see Remark \ref{rem5.4}. Below we will consider special forms of such functions.

\begin{Theorem} \label{thm:axiomatic}For any $I$ the section $c ^*w^+_I$ satisfies the following properties.
\begin{itemize}\itemsep=0pt
\item It is a meromorphic section of an admissible line bundle $p_T^*\mathcal L(N,0)\otimes \mathcal T_{k,n}$.
\item The restriction of $c ^*w^+_I$ to $Y_I$, written as a function $\C^{n+2}\to \C$ with transformation pro\-per\-ties determined by $p_T^*\mathcal L(N,0)\otimes \mathcal T_{k,n}$, is
 \begin{gather*}
 \frac {\prod\limits_{a\in I,\, b\in \bar I}\theta(z_a-z_b+\epsilon(a,b)y)} {\prod\limits_{a\in I}\theta(\lambda-(w(a,I)+1)y)},
 \end{gather*}
where $\epsilon(a,b)$ is defined in Lemma~{\rm \ref{lemma-5}} and $w(a,I)$ is defined in \eqref{eq-w}.
\item The restriction of $c ^*w^+_I$ to any $Y_J$, written as a function $\C^{n+2}\to \C$ with transformation properties determined by $p_T^*\mathcal L(N,0)\otimes \mathcal T_{k,n}$, is of the form
\begin{gather}\label{eqn:horizontal}
\frac{1}{\psi_I} \prod_{a\in J} \prod_{b\in \bar J,\, b<a} \theta(z_a-z_b+y) \cdot F_{I,J},
\end{gather}
where $F_{I,J}$ is a {\em holomorphic} function.
\end{itemize}
Moreover, these three properties uniquely determine $c ^*w^+_I$.
\end{Theorem}

\begin{Remark} From the second property one can calculate the quadratic form
\begin{gather*} %\label{eqn:NIdef}
 N_I =-2\sum_{a\in\bar I}n(a,I)z_ay-2\sum_{a\in I}z_a(\lambda+n(a,\bar I)y)\\
\hphantom{N_I =}{} +(k(n-k)-\sum_{a\in I}n(a,\bar I))y^2-\sum_{a\in I}(\lambda- (n(a,I)+1)y+n(a,\bar I)y)^2,
\end{gather*}
cf.~\eqref{e-NI}.
\end{Remark}

\begin{Remark}Lemma \ref{lemma-5} (ii) implies the triangularity property
\begin{itemize}\itemsep=0pt
\item the restriction of $c ^*w^+_I$ to $Y_J$ is 0 unless $J\leq I$.
\end{itemize}
According to Theorem \ref{thm:axiomatic} this property is a consequence of the three properties listed.
\end{Remark}

\begin{Remark}The third listed property is a local version of a support condition used in the axiomatic description of cohomological stable envelopes in \cite[Theorem 3.3.4(i)]{MO}; see also the corresponding axiom in K-theory in \cite[Theorem~3.1(I)]{RTV2}.
\end{Remark}

\begin{proof}The first two properties of $c ^*w^+_I$ are claimed in Proposition~\ref{prop-8} and Lem\-ma~\ref{lemma-5}(ii). Inspecting the explicit formula for $w^+_I$ in Section \ref{ss-expli} one finds that, after substitution $t_i$ by $z_a$, $a\in J$, all non-zero terms are of the form~(\ref{eqn:horizontal}), which proves the third property.

Now we prove that the three properties uniquely determine $c ^*w_I^+$. Let a section satisfy the three listed properties, and let $\kappa_I$ be the difference of that section and $c ^*w_I^+$.
Assume that $\kappa_I$ is not~0. Then there exists a $J$ such that $\kappa_I$ restricted to~$Y_J$ is not~0. For a total ordering $\prec$ refining the partial order $<$ on the cardinality $k$ subsets of $[n]$ let us choose $J$ to be the largest with the property $\kappa_I|_{Y_J}\not=0$. We have $J\not=I$ because of the second property.

We claim that $\kappa_I|_{Y_J}$, written as a function $\C^{n+2}\to \C$ with transformation properties determined by $p_T^*\mathcal L(N_I,0)\otimes \mathcal T_{k,n}$, is of the form
\begin{gather}\label{eqn:div1}
\frac{1}{\psi_I} \prod_{a\in J} \prod_{b\in \bar J,\, b<a} \theta(z_a-z_b+y) \cdot \prod_{a\in J} \prod_{b\in \bar J,\,b>a} \theta(z_b-z_a)\cdot F_1,
\end{gather}
where $F_1$ is holomorphic. The fact that this function can be written in the form
\begin{gather}\label{eqn:div2}
\frac{1}{\psi_I} \prod_{a\in J}\prod_{b\in\bar J,\,b<a} \theta(z_a-z_b+y) \cdot F_2,
\end{gather}
with $F_2$ holomorphic, is explicit from the third property. We need to prove that $F_2$ is the product of $\prod\limits_{a\in J}\prod\limits_{b\in\bar J,\, b>a} \theta(z_b-z_a)$ and a holomorphic function. Let $a\in J$, $b\in \bar{J}$ and $b>a$. Denote $J'=(J-\{a\})\cup \{b\}$. Observe that $J<J'$ and hence $J\prec J'$. From the choice of $J$ therefore it follows that $\kappa_I$ restricted to $Y_{J'}$ is 0. The diagonal $\Delta_{J,J'}=\{z_a=z_b\}$ is included both in~$Y_J$ and $Y_{J'}$, hence we obtain that the substitution of $z_a=z_b$ into $\kappa_I|_{Y_J}$ vanishes. It follows that the function $F_2$ vanishes on the hyperplane $z_a=z_b$ and its translates by the lattice $\Z^{n+2}+\tau\Z^{n+2}$. The zeros of $\theta(z_b-z_a)$ are exactly these hyperplanes and are of first order, therefore $F_2$ can be written as a product of $\theta(z_b-z_a)$ times a holomorphic function. Iterating this argument for all $(a,b)$ with $a\in J$, $b\in \bar{J}$, $b>a$ we obtain that~(\ref{eqn:div2}) is in fact of the form~(\ref{eqn:div1}), what we claimed.

Observe that the product of theta functions in~(\ref{eqn:div1}) is the numerator of $c ^*w_J^+|_{Y_J}$. Hence we obtain that (\ref{eqn:div1}) further equals
\begin{gather*}
\frac{\prod\limits_{a\in J} \theta(\lambda-(w(a,J)+1)y)}{\psi_I} \cdot c ^*w_J^+|_{Y_J} \cdot F_1.
\end{gather*}
Since the transformation properties of $\kappa_I|_{Y_J}$ are determined by $p_T^*\mathcal L(N_I,0)\otimes \mathcal T_{k,n}$, and those of $c ^* w_J^+|_{Y_J}$ are determined by $p_T^*\mathcal L(N_J,0)\otimes \mathcal T_{k,n}$, we have that the transformation properties of
\begin{gather}\label{eqn:div3}
\frac{\prod\limits_{a\in J} \theta(\lambda-(w(a,J)+1)y)}{\psi_I} \cdot F_1
\end{gather}
are determined by $p_T^*\mathcal L(N_I-N_J,0)$~-- the factor $\mathcal T_{k,n}$ canceled.

Let $a\in J \cap \bar{I}$, and consider (\ref{eqn:div3}) as a function of $z_a$, let us call it $f(z_a)$. Since the first factor (the fraction) only depends on $\lambda$ and $y$, $f$ is a {\em holomorphic} function of $z_a$ for generic $y,\lambda$. Comparing the $z_a$ dependence of $N_I$ and $N_J$ we obtain that
\begin{gather}\label{eqn:qq}
f(z_a+\tau)={\rm e}^{-2\pi{\rm i} (\lambda+sy)} f(z_a), \qquad f(z_a+1)=f(z_a),
\end{gather}
for some integer $s$. Using the 1-periodicity, we expand
\begin{gather*}
f(z_a)=\sum_{m\in \mathbb Z} a_m {\rm e}^{2\pi{\rm i} m z_a},
\end{gather*}
and using the first transformation property of (\ref{eqn:qq}) we obtain
\begin{gather*}
\sum_m a_m {\rm e}^{2\pi{\rm i} m z_a} \big( {\rm e}^{2\pi{\rm i} m \tau} - {\rm e}^{-2\pi{\rm i} (\lambda+sy)} \big)=0,
\end{gather*}
implying $a_m=0$ for all $m\in \mathbb Z$. We obtained $F_1=0$, and in turn, $\kappa_I|_{Y_J}=0$. This is a~contradiction proving that $\kappa_I$ is 0 on all $Y_J$.
\end{proof}

\subsection*{Acknowledgments} We thank Nora Ganter and Mikhail Kapranov for explanations on equivariant elliptic co\-ho\-mo\-logy. G.F.\ was supported in part by the National Centre of Competence in Research ``SwissMAP~-- The Mathematics of Physics'' of the Swiss National Science Foundation. R.R.\ was supported by the Simons Foundation grant~\#523882. A.V.\ was supported in part by NSF grant DMS-1362924 and Simons Foundation grant~\#336826. We thank the Forschungsinstitut f\"ur Mathematik at ETH Z\"urich and the Max Planck Institut f\"ur Mathematik, Bonn, where part of this work was done, for hospitality.

\pdfbookmark[1]{References}{ref}
\LastPageEnding

\end{document}